\documentclass[a4paper,12pt]{article}

\usepackage{amssymb}
\usepackage{bbm,hyperref,pifont,amsfonts,latexsym,amsmath,amssymb,cite,dcolumn,dcolumn,indentfirst,amsthm,mathrsfs}

\usepackage{geometry}
\numberwithin{equation}{section}

\newtheorem{thm}{Theorem}[section]
\newtheorem{lema}{Lemma}[section]

\theoremstyle{definition}\theoremstyle{plain}\theoremstyle{remark}

\newtheorem{rk}{\textbf{Remark}}
\newtheorem{theorem}{\textbf{Theorem}}

\renewcommand{\thefootnote}{\fnsymbol{footnote}}
\begin{document}

\def\thefootnote{}
\title{\bf \Large Existence of strong solutions for the compressible Ericksen-Leslie model
\author{\large{Xiangao Liu$^{*}$, Lanming Liu$^{*}$, Yihang Hao$^{*}$}\\
{\it \small  School of Mathematic Sciences, Fudan University,
Shanghai, 200433, P.R.China }}  }
 \date{}
 \maketitle
\footnote{ $^{*}$E-mail: xgliu@fudan.edu.cn \   lanmingliu@gmail.com \  10110180022@fudan.edu.cn}

\maketitle
\begin{center}
\begin{minipage}[t]{14cm}{\bf \small {Abstract:}}{\small {
 In this paper, we prove the existence and uniqueness of local strong solutions of the hydrodynamics of nematic liquid crystals system under the initial data satisfying a natural compatibility condition. Also the  global strong solutions of the system with small initial data are obtained.}


}
\end{minipage}
\end{center}


\section {Introduction}
It is well known liquid crystals are states of matter which are
capable of flow and in which the molecular arrangement gives rise to
a preferred direction. By Ericksen-Leslie theory, the compressible
nematic liquid crystals reads  the following system:
\begin{align}
  &\rho_t+\mathrm{div}(\rho u) = 0,\label{md1}\\
  &(\rho u)_t+\mathrm{div}(\rho u\otimes u)+\bigtriangledown p=\mu\bigtriangleup u-\lambda\mathrm{div}(\bigtriangledown d\odot\bigtriangledown d-\frac{1}{2}(\left |\bigtriangledown d\right|^2+F(d))I), \label{dl1} \\
  &d_t+u\cdot \bigtriangledown d=\nu (\bigtriangleup d-f(d)),\label{yj1}
\end{align}
in $\Omega\times(0,T)$, for a bounded smooth domain $\Omega\subset
\mathbb{R}^3$. $\rho(t,x)$ is density, $u(t,x)$ the velocity
field, $d(t,x)$  orientation parameter of the liquid
crystal and  $p(\rho)$ pressure with
\begin{eqnarray}
p=p(\cdot)\in C^1[0,\infty), \qquad p(0)=0.\label{yq1}\label{yq2}
\end{eqnarray}
The viscosity coefficients $\mu,\lambda,\nu$ are positive constants.
The unusual term $\bigtriangledown d\odot\bigtriangledown d$ denotes
the $3\times3$ matrix whose $(i,j)$-th element is given by
$\sum_{k=1}^3\partial_{x_i} d_k\partial_{x_j} d_k$. $I$ is the unite
matrix.
\begin{eqnarray*}
 \ f(d)=\frac{1}{\sigma^2}(|d|^2-1)d\qquad  \mathrm{and} \qquad
 F(d)=\frac{1}{4\sigma^2}(|d|^2-1)^2.
\end{eqnarray*}
We are interested in the initial data
\begin{eqnarray}
\begin{array}{llll}
 \rho(0,x)=\rho_0\geq 0, & u(0,x)=u_0, & d(0,x)=d_0(x),&\forall x\in \Omega,
 \end{array} \label{cz1}
\end{eqnarray}
and the boundary condition
\begin{eqnarray}
u(t,x)=0,\quad d(t,x)=d_0(x),\quad   |d_0(x)|=1, \quad \forall (t,x)\in (0,T)\times\partial\Omega.\label{bz1}
\end{eqnarray}

The first author \cite{liu}  firstly considers the model
(\ref{md1})-(\ref{yq1}) and gives a global existence of finite
energy weak solutions by using Lions's technique (see
\cite{lp2,feireisl1}). The incompressible model of a
similar simplified Ericksen-Leslie model has been studied in the
papers \cite{ll1, ll2,ll3,lc,coutand}. Recently, for the
incompressible Ericksen-Leslie model, Wen-Ding \cite{wen}
have gave the proof of local strong solutions in two and three
dimensions, and Lin-Lin-Wang \cite{ll0} have established the
existence of global (in time) weak solutions on a bounded smooth
domain in two dimensions .

Throughout this paper, we use the following simplified notations
\begin{eqnarray}
\begin{array}{llll}
L^q=L^q(\Omega),&W^{k,q}=W^{k,q}(\Omega),&H^k=W^{k,2}(\Omega),&H^1_0=W^{1,2}_0(\Omega).
\end{array}
\end{eqnarray}
In order to obtain strong solutions, we need $(\rho_0, u_0, d_0)$ satisfies the regularity
\begin{eqnarray}
\rho_0 \in W^{1,6},\qquad u_0 \in H_0^1\cap H^2,\qquad d_0\in H^3,\label{zz1}
\end{eqnarray}
and a compatibility condition
\begin{eqnarray}
\mu\bigtriangleup u_0-\lambda\mathrm{div}(\bigtriangledown
d_0\odot\bigtriangledown d_0-\frac{1}{2}(\left |\bigtriangledown
d_0\right|^2+F(d_0))I)-\nabla p_0=\rho_0^\frac{1}{2} g, \forall x\in \Omega,\label{xr1}
\end{eqnarray}
for some $g \in L^2$.\\
 The compatibility condition is a natural request to insure  $\|\sqrt{\rho}u_t(0)\|_{L^2}$ bounded, and a
compensation to the lack of a positive lower bound of $\rho_0$. It is firstly introduced by Salvi, Str$\breve{a}$skraba \cite{salvi} and the authors
of \cite{kim3} independently.

Our main result is the following theorem:
\begin{theorem}\label{localexistence}
{\bf Part I: (Local existence)} \,\, Under the assumptions of the
regularity condition (\ref{zz1}) and the compatibility condition
(\ref{xr1}), there exists a small $T^{*}\in (0,T)$ such that the
system (\ref{md1})-(\ref{bz1}) has a unique local strong solution
$(\rho,u,d)$ satisfying
\begin{eqnarray}
\begin{array}{ll}
\rho \in C([0,T^{*}]; W^{1,6}), &\rho _t\in C([0,T^{*}]; L^6),\\
u\in C ([0,T^{*}]; H_0^1\cap H^2)\cap L^2(0,T^{*};W^{2,6} ),&d\in C([0,T^{*}]; H^3), \\
 u_t\in L^2(0,T^{*}; H_0^1),& d_t \in C([0,T^{*}]; H^1_0)\cap L^2(0,T;H^2), \\
\sqrt{\rho} u_t \in C([0,T^{*}]; L^2),&d_{tt}\in L^2(0,T^*;L^2).
\end{array}\label{jie1}
\end{eqnarray}
{\bf Part II: (Continuity of initial data)} \,\, Suppose that
$(\widetilde{\rho},\widetilde{u},\widetilde{d})$ is another solution
of (\ref{md1})-(\ref{yq1}) with the following initial boundary
conditions
\begin{eqnarray}
\left\{\begin{array}{lll}
 \widetilde{\rho}(0,x)=\widetilde{\rho}_0\geq 0, &\widetilde{ u}(0,x)=\widetilde{u}_0,
  & \widetilde{ d}(0,x)=\widetilde{d}_0(x),\ \ \forall x\in \Omega,\\
 \widetilde{u}(t,x)=0,&\widetilde{d}(t,x)=d_0(x),
  & \forall (t,x)\in (0,T)\times\partial\Omega,
 \end{array}\right. \nonumber
\end{eqnarray}
then for any $t\in (0,T^*]$, the following quantities tend to zero
when
$(\widetilde{\rho}_0,\widetilde{u}_0,\widetilde{d}_0)\rightarrow(\rho_0,u_0,d_0)$
in $W^{1,6}\times H^2\times H^3$:
\begin{eqnarray*}
&&\|d-\widetilde{d}\|_{H^2}(t),\
\|d-\widetilde{d}\|_{L^2(0,T^*;H^3)},\
\|d_t-\widetilde{d_t}\|_{L^2}(t),\ \|d_t-\widetilde{d_t}\|_{L^2(0,T^*;H^1)},\\
&&\|u-\widetilde{u}\|_{H^1}(t),\ \|u-\widetilde{u}\|_{L^2(0,T^*;H^2)},\
\|\sqrt{\rho}(u_t-\widetilde{u_t})\|_{L^2(0,T^*;L^2)},\
\|\rho-\widetilde{\rho}\|_{L^6}(t).
\end{eqnarray*}
\end{theorem}

In general, as Navier-Stokes equations we have not the  global
existence with large initial data. However, a blow-up criterion is obtained
in our  article \cite{liu1} and for small initial data, we have the following global existence.

\begin{theorem}($\mathbf{Global\ existence\ with\ small\ initial \ data}$)\label{globalex}
Let $\alpha$ be a nonnegative constant and $m$ a constant unite
vector in $\mathbb{R}^3$. Then there exists a positive constant
$\theta_0$ small such that if the initial data satisfies
\begin{eqnarray}
\max\{\|\rho_0-\alpha\|_{W^{1,6}},\|u_0\|_{H^2},\|d_0-m\|_{H^3},
\|g\|_{L^2}^2\}\leq \theta,\label{globala}
\end{eqnarray} for all $\theta\in(0,\theta_0]$,
then the system(\ref{md1})-(\ref{bz1}) has a unique global strong
solution.
\end{theorem}
\begin{rk}
In this paper, we only consider the case $\alpha=0$, because the
other case ($\alpha>0$) can be induced to the problem with a positive initial data of
density.
\end{rk}

The methods to prove the Theorem \ref{localexistence} and Theorem \ref{globalex} are
the successive approximation in Sobolev spaces (see \cite{kim1, kim2}). We generalize this method to two variables, which is based on the following careful observation on coupling terms of $d$ and $u$:\\
(1)  $\|d\|_{H^2}^2$ can be deduced by $\int_0^t \|u\|_{H^2}^2\mathrm{d}\tau$ and $\|u\|_{H^1}^2$;\\
(2) $\|u\|_{H^2}^2$ can be derived by $\int_0^t \|d\|_{H^2}^2\mathrm{d}\tau$.\\
Since the initial density has  vacuum,  the movement equation
(\ref{dl1}) becomes a degenerate parabolic-elliptic couples system.
To overcome this difficulty, as usual, the technique is to approximate the
nonnegative initial data of density by a positive initial data. For a special linear equations of the
system (\ref{md1})-(\ref{yj1}), we prove not only the local
existence of strong solutions  with large initial data but also the
global existence of strong solutions with small initial data. Employing energy law and higher energy inequalities, we can prove both the uniqueness and the continuous dependence on the initial data.

This paper is written as follows. In section \ref{localex}, after establishing a
linear problem of the nonlinear problem (\ref{md1})-(\ref{yj1}), we
prove local existence of a strong solution to the special linear
problem with a positive initial data of density. And we also
establish some uniform a prior estimates, which imply the existence
of a local strong solution to the linear problem when the initial
data of density allows vacuum in section \ref{localexlinear}. In section \ref{It}, after constructing a sequence of approximate solutions, a strong solution
of the nonlinear problem (\ref{md1})-(\ref{yj1}) is obtained. In
section \ref{unique}, the uniqueness and continuous dependence on initial data
are proved.

\section{A linear problem}\label{localex}

\subsection{ Linearization }
At the beginning, we linearize the equations (\ref{md1})-(\ref{yj1})
as following:
\begin{eqnarray}
  &&\rho_t+\mathrm{div}(\rho v) = 0,\label{md2}\\
  &&d_t+v\cdot \bigtriangledown d=\nu (\bigtriangleup d-\frac{1}{\sigma^2}(n+m)\cdot(d-m)n),\label{yj2}\\
  &&\rho u_t+\rho v\cdot u+\bigtriangledown p(\rho)=\mu\bigtriangleup u-\lambda(\bigtriangledown d)^T(\bigtriangleup d-f(d)),\label{dl2} \label{dl22}
\end{eqnarray}
where $m$ is a constant unite vector and $(v,n)$ satisfies the following
regularity
\begin{eqnarray}
\begin{array}{cc}
v\in C([0,T]; H_0^1\cap H^2)\cap L^2(0,T; W^{2,6}),& v_t\in  L^2(0,T; H_0^1),\\
n\in C([0,T];H^2)\cap L^2(0,T; H^3),& n_t\in C([0,T]; H_0^1)\cap L^2(0,T; H^2).
\end{array}\label{as1}
\end{eqnarray}
Assume further that
\begin{eqnarray}
\left\{\begin{array}{ccc}\label{as2}
v(0,x)=u_0(x),&n(0,x)=d_0(x),&  \forall x\in \Omega,\\
v(t,x)=0,&n(t,x)=d_0(x),&  \forall (t,x)\in (0,T)\times \partial\Omega.
\end{array}\right.
\end{eqnarray}

\subsection{The existence of  approximate solutions }

For each $\delta\in(0,1)$, let $u_0^{\delta}$ be the solution to the boundary value problem:
\begin{eqnarray}
\left\{\begin{array}{l}
\mu\bigtriangleup u_0^\delta-\lambda\mathrm{div}(\bigtriangledown
 d_0\odot\bigtriangledown d_0-\frac{1}{2}(\left |\bigtriangledown
 d_0\right|^2+F(d_0))I)-p(\rho_0^\delta)=(\rho_0^\delta)^\frac{1}{2} g,\\
u_0^{\delta}=0, \forall x \in \partial \Omega,
\end{array}\right.\label{xr2}
\end{eqnarray}
where $\rho_0^\delta=\rho_0+\delta$. Then   $u_0^{\delta}\rightarrow
u_0$ in $H_0^1\cap H^2$ as $\delta \rightarrow 0$.

For the linear equations (\ref{md2})-(\ref{dl2}), we have the following theorem.

\begin{thm}\label{j1}
Assume $\rho_0,\  u_0$ and $ d_0$  satisfy the regularity
(\ref{zz1}) and $(v, n)$ satisfies the above conditions (\ref{as1})
and (\ref{as2}). Then there exists a unique strong solution $(\rho,
u,d)$ of the linear equations (\ref{md2})-(\ref{dl2}) with the
initial data $(\rho_0^\delta, u_0^\delta, d_0)$ and the boundary
condition (\ref{bz1}) such that
\begin{eqnarray}
\begin{array}{ll}
\rho\in C([0,T]; W^{1,6}),& \rho_t\in C([0,T]; L^6),\\
u\in C([0,T]; H_0^1\cap H^2)\cap L^2(0,T; W^{2,6}),& u_t\in C([0,T]; L^2)\cap L^2(0,T; H_0^1),\\
d\in C([0,T];H^2)\cap L^2(0,T; H^3),& d_t\in C([0,T]; H_0^1)\cap L^2(0,T; H^2),\\
u_{tt}\in L^2(0,T;H^{-1}),&d_{tt}\in L^2(0,T;L^{2}).
\end{array}\label{zz2}
\end{eqnarray}
\end{thm}

We will prove the theorem  by the following three lemmas. Firstly,
suppose the constants $c_0,\ c_1$ and $c_2$ satisfying
\begin{eqnarray}
c_0&> &1+\|\rho_0\|_{W^{1,6}}+\|d_0\|_{H^3}+\|u_0\|_{H^2}+\|g\|_{L^2},\label{cs0}\\
c_1&>& \sup_{0\leq t\leq T}(\|v\|_{H^1_0}+\|n\|_{H^1}+\|n_t\|_{H^1_0})+\int_0^T(\|\bigtriangledown v_t\|^2_{L^2}+\|v\|^2_{W^{2,6}}\nonumber\\
&&+\|\bigtriangledown^2 n_t\|^2_{L^2}+\|n\|^2_{H^3})\mathrm{d}t,\label{cs1}\\
c_2&>&\sup_{0\leq t\leq T}(\|\bigtriangledown ^2 v\|_{L^2}+\|\bigtriangledown ^2 n\|_{L^2}),\\
c_2&>&c_1>c_0>1.\label{cs2}
\end{eqnarray}

\begin{lema}\label{lm1}
Assume $\rho_0$ and $v$ satisfy the regularities (\ref{zz1}) and
(\ref{as1}) respectively. Then the problem (\ref{md2}) and
(\ref{cz1}) has a global unique strong solution such that
\begin{eqnarray}
\begin{array}{cc}
\rho\in C([0,T]; W^{1,6}),&\rho_t\in C ([0,T]; L^6).
\end{array}\label{mdc}\label{md4a}
\end{eqnarray}
Moreover, if $v, n$ satisfy (\ref{cs0})-(\ref{cs2}), then there is a
$T_1=\min \{c_1^{-1}, T\}$ such that $\forall t\in [0,T_1],$
\begin{eqnarray}
\begin{array}{cc}
\|\rho\|_{W^{1,6}}(t)\leq Cc_0 ,& \|\rho_t\|_{L^6}(t)\leq
Cc_0c_2.
\end{array}\label{md4c}
\end{eqnarray}
In particular,
\begin{eqnarray}
\begin{array}{ccc}
\|p\|_{L^6}(t)\leq CM(c_0),& \|\bigtriangledown p\|_{L^6}(t)\leq
CM(c_0)c_0,& \| p_t\|_{L^6}(t)\leq CM(c_0)c_0c_2,
\end{array}\label{yq3}
\end{eqnarray}
where the constant $M(c_0)$ is defined by (\ref{yqcs}).
\end{lema}
\begin{rk}
The above lemma is also true to the problem (\ref{md2}) with the initial data $\rho_0^\delta$ and the estimates \ref{md4c}-\ref{yq3} also hold for all small $\delta$.
\end{rk}
\begin{proof}
 Now let's start to prove Lemma \ref{lm1}. The existence is obvious, based on the classical method of characteristics (see \cite{kim1}).
 So
\begin{eqnarray}
\rho(t,x)=\rho_0(y(0,t,x))\exp(-\int_0^t \mathrm{div} v(\tau, y(\tau, t,x))\mathrm{d}\tau),\label{mdb}
\end{eqnarray}
where $y(\tau,t,x)(\in C([0,T]\times [0,T]\times
\overline{\Omega}))$ is a solution of the initial value problem:
\begin{eqnarray}
\left\{\begin{array}{l}
\frac{\partial}{\partial \tau}y(\tau,t,x)=v(\tau,y(\tau,t,x)), \qquad 0\leq \tau \leq T,\nonumber\\
y(t,t,x)=x, \qquad\qquad\qquad 0\leq t \leq T,\ x\in \overline{\Omega}.
\end{array}\right.\nonumber
\end{eqnarray}
Moreover, we have
\begin{eqnarray}
\|\rho\|_{W^{1,6}}(t)\leq \|\rho_0\|_{W^{1,6}}\exp(\int_0^t \|\bigtriangledown v\|_{W^{1,6}}(\tau)\mathrm{d}\tau),\qquad \forall t\in[0,T
].\label{mdgj1a}
\end{eqnarray}
So that for all $t\in [0,T_1],$
\begin{eqnarray}
&\|\rho\|_{W^{1,6}}(t)&\leq \|\rho_0\|_{W^{1,6}}\exp(\int_0^t \|\bigtriangledown v\|_{W^{1,6}}(\tau)\mathrm{d}\tau)\nonumber\\
&& \leq c_0\exp(t^\frac{1}{2}(\int_0^t \|\bigtriangledown v\|^2_{W^{1,6}}(\tau)\mathrm{d}\tau)^\frac{1}{2})\nonumber\\
&&\leq ec_0.\label{md4ba}
\end{eqnarray}
Because of $\|\rho\|_{L^\infty}\leq
\widetilde{C}\|\rho\|_{W^{1,6}}$, set
\begin{eqnarray}
M(c_0)=\sup_{0\leq \cdot\leq
\widetilde{C}ec_0}(1+p(\cdot)+p'(\cdot)).\label{yqcs}
\end{eqnarray}
We can obtain for all $t\in [0,T_1],$
\begin{eqnarray}
\|\rho_t\|_{L^6}(t)&=&\|-v\cdot \bigtriangledown \rho-\rho \mathrm{div}v\|_{L^6}\nonumber\\
&\leq& \|v\|_{L^\infty}\|\bigtriangledown \rho\|_{L^6}+\|\rho\|_{L^\infty}\|\bigtriangledown v\|_{L^6}\nonumber\\
&\leq&\|v\|_{H^2}\| \rho\|_{W^{1,6}}+\|\rho\|_{W^{1,6}}\| v\|_{H^2}\nonumber\\
&\leq& Cc_0c_2.\nonumber
\end{eqnarray}
Similarly,  for all $t\in [0,T_1],$
\begin{eqnarray}
\begin{array}{ll}
\|p\|_{L^6}(t)\leq CM(c_0),&
\|\bigtriangledown p\|_{L^6}(t) \leq CM(c_0)c_0,\\
\| p_t\|_{L^6}(t)\leq CM(c_0)c_2c_0.&
\end{array}\nonumber
\end{eqnarray}

\end{proof}

\begin{lema}\label{lm2}
Under the hypotheses of Theorem \ref{j1}, then the equation
(\ref{yj2}) with the initial boundary conditions
(\ref{cz1})-(\ref{bz1}) has a global unique strong solution $d$
satisfying
(\ref{zz2}).\\
Moreover, if $v, n$ satisfy  (\ref{cs0})-(\ref{cs2}), then there is a
$T_3=\min \{c_2^{-22}, T\}$ such that
\begin{eqnarray}
&&\sup_{0\leq t\leq T_3}(\|d\|_{H^1}+c_1^{-2}\|\bigtriangledown^2 d\|_{L^2}+c_0^\frac{3}{2}\|d_t\|_{H^1}+c_2^{-2}c_1^{-1}\|\bigtriangledown d\|_{H^2})\nonumber\\
&&+\int_0^{T_3} \|d_t\|_{H^2}^2+c_0^3\|d\|^2_{H^3}\mathrm{d}t\leq Cc_0^3 . \label{lm2a}
\end{eqnarray}
\end{lema}
\begin{proof}
 By the classic Galerkin method to the linear parabolic equation
 (\ref{yj2}) with (\ref{cz1})-(\ref{bz1}), the existence and regularity of $d$ described in (\ref{zz2})  can be obtained.

Differentiating (\ref{yj2}) with respect to time, multiplying by $d_t$ and then integrating over $\Omega$, we can deduce that
\begin{align}
&\frac{1}{2}\frac{\mathrm{d}}{\mathrm{d}t}\int_{\Omega} |d_t|^2\mathrm{d}x+\nu\int_{\Omega} |\bigtriangledown d_t|^2\mathrm{d}x\nonumber\\
\leq& C(\|v_t\|_{L^6}\|\bigtriangledown d\|_{L^2}\|d_t\|_{L^3}+\|v\|_{L^\infty}\|\bigtriangledown d_t\|_{L^2}\|d_t\|_{L^2}+\|n\|_{L^3}\|n_t\|_{L^6}\|d-m\|_{L^3}\|d_t\|_{L^6}\nonumber\\
&\ \ \ +\|n+m\|_{L^3}\|n_t\|_{L^6}\|d-m\|_{L^3}\|d_t\|_{L^6}+\|n\|_{L^6}\|n+m\|_{L^6}\|d_t\|^2_{L^3})\nonumber\\
=&\sum_{i=1}^5 I_i.\label{d1}
\end{align}
Here
\begin{eqnarray}
&I_1&\leq C \eta \|v_t\|^2_{H^1}\|\bigtriangledown d\|^2_{L^2}+C\eta^{-1} \|d_t\|^2_{L^2}+\frac{\nu}{5}\|\bigtriangledown d_t\|^2_{L^2},\nonumber\\
&I_2&\leq  C\|v\|^2_{H^2}\|d_t\|^2_{L^2}+\frac{\nu}{5}\|\bigtriangledown d_t\|^2_{L^2},\nonumber\\
&I_3+I_4&\leq C((\|n\|^2_{L^2}+1)(\|n\|^2_{H^1}+1)\|n_t\|^4_{L^6}\|d-m\|^2_{L^2}+\|\bigtriangledown d\|^2_{L^2}+\|d-m\|^2_{L^2})\nonumber\\
&&\quad+\frac{\nu}{5}\|\bigtriangledown d_t\|^2_{L^2},\nonumber\\
&I_5&\leq C\|n\|^2_{H^1}\|n+m\|^2_{H^1}\|d_t\|^2_{L^2}+\frac{\nu}{5}\|\bigtriangledown d_t\|^2_{L^2},\nonumber
\end{eqnarray}
where the small positive constant $\eta$ will be fixed later.\\
Since
\begin{align}
&\frac{\mathrm{d}}{\mathrm{d}t}\int_\Omega |\bigtriangledown d|^2\mathrm{d}x=2\int_\Omega \bigtriangledown d:\bigtriangledown d_t\mathrm{d}x
\leq \int_\Omega |\bigtriangledown d|^2\mathrm{d}x+\int_\Omega|\bigtriangledown d_t|^2\mathrm{d}x,\label{d4}\\
&\frac{\mathrm{d}}{\mathrm{d}t}\int_\Omega |d-m|^2\mathrm{d}x \leq
\int_\Omega | d-m|^2\mathrm{d}x+\int_\Omega|
d_t|^2\mathrm{d}x,\label{d5}
\end{align}
combining (\ref{d1}), (\ref{d4}) and (\ref{d5}), we get
\begin{align}
&\frac{\mathrm{d}}{\mathrm{d}t}\int_{\Omega} |d_t|^2+|\bigtriangledown d|^2+|d-m|^2\mathrm{d}x+\int_{\Omega} |\bigtriangledown d_t|^2\mathrm{d}x\nonumber\\
\leq&C(\|\bigtriangledown d\|^2_{L^2}+\|d-m\|^2_{L^2}+\| d_t\|^2_{L^2})\cdot(\eta^{-1}+\|v\|^2_{H^2}+\|n\|^2_{H^1}\|n+m\|^2_{H^1}\nonumber\\
&+\eta \|v_t\|^2_{H^1} +(\|n\|^2_{L^2}+1)(\|n\|^2_{H^1}+1)\|\bigtriangledown n_t\|^4_{L^2}+1)
.
\label{d6}
\end{align}
From the equation (\ref{yj2}), we can deduce
\begin{align}
\|d_t\|_{L^2}(0)\leq& C(\|\bigtriangleup d_0\|_{L^2}+\|u_0\|_{H^2}\|\bigtriangledown d_0\|_{L^2}+\|d_0+m\|_{L^2}\|d_0-m\|_{L^2}\|d_0\|_{L^2})\nonumber\\
\leq& Cc_0^3.\label{d6a}
\end{align}
Hence, by Gronwall's inequality, we can deduce from (\ref{d6})
\begin{eqnarray}
&&\int_{\Omega} |d_t|^2+|\bigtriangledown d|^2+|d-m|^2\mathrm{d}x+\int_0^t\!\!\!\int_{\Omega} |\bigtriangledown d_t|^2\mathrm{d}x\mathrm{d}\tau\nonumber\\
&\leq&Cc_0^6\exp(C\int_0^t\eta^{-1}+\|v\|^2_{H^2}+\|n\|^2_{H^1}+\|n+m\|^2_{H^1}+\eta \|v_t\|^2_{H^1}\nonumber\\
&&+(\|n\|^2_{L^2}+1)(\|n\|^2_{H^1}+1)\|\bigtriangledown n_t\|^4_{L^2}\mathrm{d}\tau).\label{d7}
\end{eqnarray}
Taking $\eta =c_1^{-1}$ and using the assumption (\ref{cs0})-(\ref{cs2}), we obtain
\begin{eqnarray}
\sup_{0\leq t\leq T_2}\int_{\Omega} |d_t|^2+|\bigtriangledown d|^2+|d-m|^2\mathrm{d}x+\int_0^{T_2}\!\!\!\int_{\Omega} |\bigtriangledown d_t|^2\mathrm{d}x\mathrm{d}\tau\leq Cc_0^6,\ \label{d77}
\end{eqnarray}
where $T_2=\min\{c_2^{-8},T_1\}.$

From the equation (\ref{yj2}) and using the elliptic estimates, we get
\begin{eqnarray}
\| d-m\|_{H^2}
&\leq& C (\|d_0-m\|_{H^2}+\|d_t\|_{L^2}+\|v\cdot \bigtriangledown d \|_{L^2}+\|(n+m)\cdot(d-m)n\|_{L^2})\nonumber\\
&\leq& C(c_0^3+\|v\|_{H^1}\|\bigtriangledown d \|^\frac{1}{2}_{L^2}\|d-m\|^\frac{1}{2}_{H^2}+\|n+m\|_{L^6}\|d-m\|_{L^6}\|n\|_{L^6})\nonumber\\
&\leq& C(c_0^3+\|v\|^2_{H^1}\|\bigtriangledown d \|_{L^2}+\|n+m\|_{H^1}\|d-m\|_{H^1}\|n\|_{H^1})+\frac{1}{2}\|d\|_{H^2}.\nonumber
\end{eqnarray}
So
\begin{eqnarray}
\|\bigtriangledown^2 d\|_{L^2}\leq C(c_0^3+c_1^2c_0^3+c_1^2c_0^3)\leq Cc_1^2c_0^3.\label{d8}
\end{eqnarray}

Applying the operator $\bigtriangledown$ to the linear equation (\ref{yj2}), we get
\begin{eqnarray}
\nu\bigtriangleup (\bigtriangledown d)=\bigtriangledown d_t+\bigtriangledown (v\cdot \bigtriangledown d)+\frac{\nu}{\sigma^2}\bigtriangledown((n+m)\cdot(d-m)n).\label{yj22}
\end{eqnarray}
By the elliptic estimates, we can estimate the term $\|\bigtriangledown (d-m)\|_{H^2}$ as follows
\begin{align}
\|\bigtriangledown  (d-m)\|_{H^2}\leq& C(\|\bigtriangledown d_t\|_{L^2}+\|\bigtriangledown(v\cdot \bigtriangledown d)\|_{L^2}+\|\frac{\nu}{\sigma^2}\bigtriangledown[(n+m)\cdot(d-m)n]\|_{L^2}\nonumber\\
&+ \|d_0-m\|_{H^3})\nonumber\\
\leq& C(\|\bigtriangledown d_t\|_{L^2}+\|\bigtriangledown v\|_{H^1}\|\bigtriangledown d\|_{L^2}^\frac{1}{2}\|\bigtriangledown d\|_{H^1}^\frac{1}{2}+\|v\|_{H^2}\|\bigtriangledown^2 d\|_{L^2}\nonumber\\
&\quad+\|\bigtriangledown n\|_{H^1}\|n\|_{H^1}\|d-m\|_{H^1}+\|n\|_{H^2}^2\|\bigtriangledown d\|_{L^2}+ \|d_0-m\|_{H^3})\nonumber\\
\leq& C(\|\bigtriangledown d_t\|_{L^2}+c_2^2c_1c_0^3),\label{d10a}
\end{align}
where we use the assumption (\ref{cs0})-(\ref{cs2}) and
(\ref{d77})-(\ref{d8}).

Differentiating (\ref{yj2}) with respect to time and taking inner product with $\bigtriangleup d_t$, then we can derive
\begin{eqnarray}
&&\frac{1}{2}\frac{\mathrm{d}}{\mathrm{d}t}\int_\Omega |\bigtriangledown d_t|^2\mathrm{d}x+ \frac{\nu}{C} \| d_t\|^2_{H^2}\nonumber\\
&\leq&\frac{1}{2}\frac{\mathrm{d}}{\mathrm{d}t}\int_\Omega |\bigtriangledown d_t|^2\mathrm{d}x+\nu\int_\Omega |\bigtriangleup d_t|^2\mathrm{d}x\nonumber\\
&=&\int_\Omega (v_t\cdot \bigtriangledown d)\bigtriangleup d_t\mathrm{d}x+\int_\Omega (v\cdot \bigtriangledown d_t)\bigtriangleup d_t\mathrm{d}x+\frac{\nu}{\sigma^2}\int_\Omega (n_t\cdot (d-m))n\bigtriangleup d_t\mathrm{d}x\nonumber\\
&&+\frac{\nu}{\sigma^2}\int_\Omega (n\cdot d_t)n\bigtriangleup d_t\mathrm{d}x+\frac{\nu}{\sigma^2}\int_\Omega ((n+m)\cdot(d-m))n_t\bigtriangleup d_t\mathrm{d}x\nonumber\\
&=&\sum_{j=1}^5J_j,\label{d9}
\end{eqnarray}
where we use the elliptic estimate $\|d_t\|_{H^2}^2\leq C\|\bigtriangleup d_t\|^2_{L^2}$. \\
Here
\begin{eqnarray}
|J_1|&=&|\int_\Omega (v_t\cdot \bigtriangledown d) \bigtriangleup d_t\mathrm{d}x|\nonumber\\
&\leq& \|\bigtriangledown v_t\|_{L^2}\|\bigtriangledown d\|_{L^6}\| \bigtriangledown d_t\|_{L^3}+\|v_t\|_{L^3}\|\bigtriangledown^2 d\|_{L^6}\| \bigtriangledown d_t\|_{L^2}\nonumber\\
&\leq& \eta\|\bigtriangledown v_t\|^2_{L^2}+\eta^{-1}\|\bigtriangledown d\|^2_{H^1}\| \bigtriangledown d_t\|_{L^2}\| \bigtriangledown d_t\|_{H^1}+\eta\|v_t\|^2_{L^3}\|\bigtriangledown d_t\|^2_{L^2}\nonumber\\
&&+\eta^{-1}\|\bigtriangledown^2 d\|^2_{L^6}\nonumber\\
&\leq& \eta\|\bigtriangledown v_t\|^2_{L^2}+C\eta^{-2}\varepsilon ^{-1}c_1^8c_0^{12}\| \bigtriangledown d_t\|^2_{L^2}+\varepsilon\|d_t\|^2_{H^2}+\eta\|v_t\|^2_{H^1}\|\bigtriangledown d_t\|^2_{L^2}\nonumber\\
&&+C\eta^{-1}(\|\bigtriangledown d_t\|^2_{L^2}+c_2^4c_1^2c_0^6),\nonumber\end{eqnarray}
\begin{eqnarray}
|J_2|&=&|\int_\Omega (v\cdot \bigtriangledown d_t)\bigtriangleup d_t\mathrm{d}x|\leq C\varepsilon^{-1}c_2^2\|\bigtriangledown d_t\|^2_{L^2}+\varepsilon\|\bigtriangleup d_t\|^2_{L^2},\nonumber\\
|J_4|&\leq&C\varepsilon^{-1}\|n\|_{L^\infty}^4\|d_t\|_{L^2}^2+\varepsilon \|\bigtriangleup d_t\|^2_{L^2}\leq C\varepsilon^{-1}c_2^4c_0^6+\varepsilon \|\bigtriangleup d_t\|^2_{L^2},\qquad\qquad\qquad\qquad \qquad \qquad \nonumber\\
|J_3|+|J_5|&\leq&C\varepsilon^{-1}(\|n\|^2_{L^\infty}+1)\|n_t\|^2_{L^2}\|d-m\|^2_{L^\infty}+\varepsilon \|\bigtriangleup d_t\|_{L^2}^2\nonumber\\
&\leq& C\varepsilon^{-1}c_2^2c_1^6c_0^6+\varepsilon \|\bigtriangleup d_t\|_{L^2}^2,\nonumber
\end{eqnarray}
where the small positive constants $\varepsilon$ and $\eta$ will be fixed later.\\
Substituting the above $J_1-J_5$ into (\ref{d9}) and taking $\varepsilon$ small enough, we get
\begin{eqnarray}
&&\frac{\mathrm{d}}{\mathrm{d}t}\int_\Omega |\bigtriangledown d_t|^2\mathrm{d}x+  \| d_t\|_{H^2}^2\leq A_\eta(t)\| \bigtriangledown d_t\|_{L^2}^2+B_\eta(t),\label{d9e}
\end{eqnarray}
where
\begin{eqnarray}
A_\eta(t)&=&C(\eta^{-2}c_1^8c_0^{12}+\eta\|v_t\|^2_{H^1}
+\eta^{-1}+c_2^2),\nonumber\\
B_\eta(t)&=&C(\eta\|v_t\|^2_{H^1}+\eta^{-1}c_2^4c_1^2c_0^8+\eta^{-1}c_2^2c_1^6c_0^3+
c_2^4c_0^6).\nonumber
\end{eqnarray}
In view of the inequalities (\ref{d77})-(\ref{d8}) and taking $\eta = c_2^{-1}$, we have for all $t\in[0,T_2],$
\begin{eqnarray}
\begin{array}{ll}
\int_0^{t}A_\eta(s)\mathrm{d}s\leq C+Cc_2^2c_1^8c_0^{12}t,&
\int_0^{t}B_\eta(s)\mathrm{d}s\leq C+Cc_2^5c_1^4c_0^6t.
\end{array}\label{d9e1}\label{d9e2}
\end{eqnarray}
From the equation (\ref{yj22}), we can estimate the term $\|\bigtriangledown d_t(\tau)\|_{L^2}$ at time $0$
 \begin{eqnarray}
\|\bigtriangledown d_t(0)\|_{L^2} &\leq&C(\|d_0-m\|_{H^3}+ \|\bigtriangledown u_0\|_{L^3}\|\bigtriangledown ^2d_0\|_{L^6}+\|u_0\|_{L^3}\|\bigtriangledown^2d_0\|_{L^6}
\nonumber\\
 &&+\|\bigtriangledown d_0\|_{L^6}\|d_0-m\|_{L^6}\|d_0\|_{L^6}+\|d_0\|^2_{L^6}\|\bigtriangledown d_0\|_{L^6})\nonumber\\
 &\leq&Cc_0^3.\label{d9h}
 \end{eqnarray}
Applying Gronwall's inequality to (\ref{d9e}) and using (\ref{d9e1})-(\ref{d9h}), we can obtain
\begin{align}
\sup_{0\leq t\leq T_3}\int_\Omega |\bigtriangledown d_t|^2\mathrm{d}x+\int_0^{T_3} \| d_t\|_{H^2}^2\mathrm{d}t\leq& C(c_0^3+c_2^5c_1^4c_0^6T_3)\exp{(C+Cc_2^2c_1^8c_0^{12} T_3)}\nonumber\\
\leq& Cc_0^3.\label{d9f}
\end{align}
From (\ref{d10a}) and (\ref{d9f}), we have
\begin{eqnarray}
\|\bigtriangledown d\|_{H^2}(t)\leq C(c_0^\frac{3}{2}+c_2^2c_1
c_0^3), \,\,\forall t\in[0,T_3]. \label{dd3}
\end{eqnarray}
Hence
\begin{eqnarray}
\int_0^t \|d\|^2_{H^3}\mathrm{d}\tau\leq C, \qquad  \forall t\in [0,T_3].\label{d11}
\end{eqnarray}
Since (\ref{d77}), (\ref{d8}), (\ref{d9f}) and (\ref{d11}) can deduce the estimate (\ref{lm2a}), we complete the proof.
\end{proof}

\begin{lema}\label{lm3}
Under the hypotheses of Theorem \ref{j1}, then there exists a global
unique solution $u$ of the equation (\ref{dl2}) with the initial
data $u_0^\delta$ and the boundary condition (\ref{bz1}), and the
solution $u$ has the regularity in (\ref{zz2}).\\
Moreover, if $v,n$
satisfy (\ref{cs0})-(\ref{cs2}), then, for all small $\delta$,
\begin{eqnarray}
&&\sup_{0\leq t\leq T_3}(M(c_0)c_0^{13}\|u\|_{H^1_0}+M(c_0)c_0^{8}c_1^{-6}
\|\bigtriangledown ^2 u\|_{L^2}+M(c_0)c_0^{13}\|\sqrt{\rho}u_t\|_{L^2})\nonumber\\
&&+\int_0^{T_3}c_0^{8}\|\bigtriangledown u_t\|^2_{L^2}+\|u\|^2_{W^{2,6}}\mathrm{d}t\leq Cc_0^{18}M^2(c_0).\label{lm3a}
\end{eqnarray}
\end{lema}
\begin{proof}
Because $\rho_0^\delta\geq \delta> 0$, it follows from the representation (\ref{mdb}) that
 \begin{eqnarray}
\rho(t,x)\geq\delta\exp(-\int_0^t |\bigtriangledown v(\tau)|_{W^{1,6}}\mathrm{d}\tau)\geq \underline{\delta},\qquad \forall (t,x)\in [0,T]\times \overline{\Omega},\label{mdd}
\end{eqnarray}
where $\underline{\delta}$ is a positive constant.\\
Thanks to (\ref{mdd}), we change (\ref{dl2}) into the following form:
\begin{eqnarray}
 u_t+ v\cdot \bigtriangledown u+\frac{1}{\rho}\bigtriangledown p=\frac{\mu}{\rho}\bigtriangleup u-\frac{\lambda}{\rho}(\bigtriangledown d)^{T}(\bigtriangleup d-f(d)). \nonumber
\end{eqnarray}
Applying the Galerkin method again to the above equation with the
initial data $u_0^\delta$ and the boundary condition (\ref{bz1}), we
can deduce the existence and regularity of $u$ described in
(\ref{zz2}).

Differentiating (\ref{dl22}) with respect to time $t$, multiplying by $u_t$ and then integrating over $\Omega$, we can derive
\begin{eqnarray}
&&\frac{1}{2}\frac{\mathrm{d}}{\mathrm{d}t}\int_\Omega \rho |u_t|^2\mathrm{d}x+\mu\int_\Omega |\bigtriangledown u_t|^2\mathrm{d}x\nonumber\\
&=&\int_\Omega (-\bigtriangledown p_t-\rho_t v\cdot \bigtriangledown u-2\rho  v\cdot \bigtriangledown u_t-\rho v_t\cdot \bigtriangledown u))u_t\mathrm{d}x\nonumber\\
&&-\lambda\int_\Omega  (\bigtriangledown d_t)^T(\bigtriangleup d-f(d))u_t\mathrm{d}x-\lambda\int_\Omega  (\bigtriangledown d)^T(\bigtriangleup d-f(d))_tu_t\mathrm{d}x\nonumber\\
&=&\sum_{k=1}^6 K_k.\label{u1}
\end{eqnarray}
Here
\begin{align}
|K_1|&\leq C\|p_t\|_{L^2}^2+\frac{\mu}{8}\|\bigtriangledown u_t\|_{L^2}^2\leq CM^2(c_0)c_0^2c_2^2+\frac{\mu}{8}\|\bigtriangledown u_t\|_{L^2}^2,\nonumber\\
|K_2|&\leq C\|\rho_t\|_{L^6}^2\|v\|^2_{L^6}\|\bigtriangledown u\|^2_{L^2}+\frac{\mu}{8}\|\bigtriangledown u_t\|_{L^2}^2\leq Cc_2^2c_1^2c_0^2\|\bigtriangledown u\|^2_{L^2}+\frac{\mu}{8}\|\bigtriangledown u_t\|_{L^2}^2,\ \ \ \nonumber
\end{align}
\begin{align}
|K_3|&\leq C\|\rho\|_{L^\infty}\|v\|^2_{L^\infty}\|\sqrt{\rho}u_t\|^2_{L^2}+\frac{\mu}{8}\|\bigtriangledown u_t\|_{L^2}^2\leq Cc_2^2c_0\|\sqrt{\rho}u_t\|^2_{L^2}+\frac{\mu}{8}\|\bigtriangledown u_t\|_{L^2}^2,\nonumber\\
|K_4|&\leq\eta \|\bigtriangledown v_t\|^2_{L^2}\|\sqrt{\rho}u_t\|^2_{L^2}+\eta^{-1}\|\rho\|_{L^\infty}\|\bigtriangledown u\|_{L^2}\|\bigtriangledown u\|_{H^1}\nonumber\\
& \leq\eta \|\bigtriangledown v_t\|^2_{L^2}\|\sqrt{\rho}u_t\|^2_{L^2}+\eta^{-2}c_0^2\|\bigtriangledown u\|^2_{L^2}+\|\bigtriangledown u\|^2_{H^1},\nonumber
\end{align}
\begin{eqnarray}
|K_5|&\leq&C( \|\bigtriangledown^2 d\|_{L^3}\|\bigtriangledown d_t\|_{L^2}\|u_t\|_{L^6}+\|\bigtriangledown d_t\|_{L^2}\|d\|_{L^3}\|u_t\|_{L^6}\|d\|_{L^\infty}^2\nonumber\\
&&+\|\bigtriangledown d_t\|_{L^2}\|d\|_{L^3}\|u_t\|_{L^6})\nonumber\\
&\leq&C\|\bigtriangledown^2 d\|_{L^2}\|\bigtriangledown ^2 d\|_{L^6}\|\bigtriangledown d_t\|^2_{L^2} +C\|\bigtriangledown d_t\|^2_{L^2}\|d\|_{L^2}\|d\|_{H^1}\|d\|_{H^2}^4\qquad \ \  \ \nonumber\\
&&+C\|\bigtriangledown d_t\|^2_{L^2}\|d\|_{L^2}\|d\|_{H^1}+\frac{\mu}{8\lambda}\|\bigtriangledown u_t\|^2_{L^2}\nonumber\\
&\leq& Cc_2^2c_1^6c_0^{21}+\frac{\mu}{8\lambda}\|\bigtriangledown u_t\|^2_{L^2},\nonumber
\end{eqnarray}
\begin{align}
|K_6|\leq&C(\|\bigtriangledown^2 d\|_{L^3}\|\bigtriangledown d_t\|_{L^2}\|u_t\|_{L^6}+\|\bigtriangledown d\|_{L^\infty}\|\bigtriangledown d_t\|_{L^2}\|\bigtriangledown u_t\|_{L^2}\nonumber\\
&+\|\bigtriangledown d\|_{L^2}\|d_t\|_{L^3}\|u_t\|_{L^6}+\|\bigtriangledown d\|_{L^3}\|d_t\|_{L^6}\|d\|_{L^6}^2\|u_t\|_{L^6})\nonumber\\
\leq&C(\|\bigtriangledown^2 d\|_{L^2}\|\bigtriangledown ^2 d\|_{L^6}\|\bigtriangledown d_t\|^2_{L^2}+ \|\bigtriangledown  d\|^2_{W^{1,6}}\|\bigtriangledown d_t\|^2_{L^2}\nonumber\\
&+\| d_t\|_{H^1}\|d_t\|_{L^2}\|\bigtriangledown d\|^2_{L^2}+\|\bigtriangledown d\|_{L^2}\|\bigtriangledown d\|_{H^1}\|\bigtriangledown d_t\|^2_{L^2}\|d\|^4_{H^1})\qquad\qquad\nonumber\\
&+\frac{\mu}{8\lambda}\|\bigtriangledown u_t\|^2_{L^2}\nonumber\\
\leq&Cc_2^4c_1^2c_0^{17}+\frac{\mu}{8\lambda}\|\bigtriangledown u_t\|^2_{L^2}.\nonumber
\end{align}
On the other hand
\begin{eqnarray}
\frac{\mathrm{d}}{\mathrm{d}t}\|\bigtriangledown u\|^2_{L^2}=\int_\Omega\bigtriangledown u:\bigtriangledown u_t \mathrm{d}x\leq C\|\bigtriangledown u\|^2_{L^2}+\frac{\mu}{8}\|\bigtriangledown u_t\|^2_{L^2}.\label{u2}
\end{eqnarray}
Substituting $|K_1|-|K_6|$ into (\ref{u1}) and combing with
(\ref{u2}), we  obtain
\begin{eqnarray}
&&\frac{\mathrm{d}}{\mathrm{d}t}\int_\Omega (\rho |u_t|^2+|\bigtriangledown u|^2)\mathrm{d}x+\int_\Omega |\bigtriangledown u_t|^2\mathrm{d}x\nonumber\\
&\leq& \mathscr{A}_\eta(t)(\|\sqrt{\rho}u_t\|^2_{L^2}+\|\bigtriangledown u\|^2_{L^2})+\mathscr{B}(t)+C\|\bigtriangledown u\|^2_{H^1},\label{u1g}
\end{eqnarray}
where
\begin{eqnarray}
\begin{array}{ll}
\mathscr{A}_\eta(t)=C(c_2^2c_1^2c_0^2+\eta\|\bigtriangledown v_t\|^2_{L^2}+\eta^{-2}c_0^2),&
\mathscr{B}(t)=Cc_2^4c_1^4c_0^{21}M^2(c_0).
\end{array}\nonumber
\end{eqnarray}
Taking $\eta =c_2^{-1}$ and using the estimates (\ref{md4c})-(\ref{yq3}) and (\ref{lm2a}), we have  $\ \forall t\in [0,T_3],$
\begin{eqnarray}
\begin{array}{ll}
\int_0^t\mathscr{A}_\eta(s)\mathrm{d}s \leq C+Cc_2^4t,&
\int_0^t\mathscr{B}(s)\mathrm{d}s\leq Cc_2^4c_1^4c_0^{21}M^2(c_0)t+C.
\end{array}\label{u1g1}\label{u1g2}
\end{eqnarray}
Multiplying (\ref{dl2}) by $u_t$, then integrating it over $\Omega$ and using the Young's inequality, we can obtain
\begin{eqnarray}
\int_\Omega \rho |u_t|^2\mathrm{d}x(\tau)& \leq &C\int_\Omega \rho |v|^2|\bigtriangledown u|^2+|\rho^{-\frac{1}{2}}(\mu\bigtriangleup u-\lambda\mathrm{div}(\bigtriangledown d\odot\bigtriangledown d\nonumber\\
&&\qquad-\frac{1}{2}( |\bigtriangledown d|^2+F(d))I)-\nabla p)|^2\mathrm{d}x(\tau).\nonumber
\end{eqnarray}
Hence
\begin{eqnarray}
\limsup_{\tau\rightarrow 0+}\int \rho |u_t|^2\mathrm{d}x(\tau)\leq C(c_0^5+\|g\|_{L^2}^2)\leq  Cc_0^5.\label{cz2}
\end{eqnarray}
Integrating (\ref{u1g}) with respect to time $(\tau,t)$ and letting
$\tau\rightarrow 0+$, thanks to (\ref{cz2}) and Gronwall's
inequality and the inequality (\ref{u1g1}), we have for all $t\in[0,T_3],$
\begin{eqnarray}
&&\int_\Omega (\rho |u_t|^2+|\bigtriangledown u|^2)\mathrm{d}x(t)+\int_0^{t}\!\!\!\int_\Omega |\bigtriangledown u_t|^2\mathrm{d}x\mathrm{d}\tau\nonumber\\
&\leq& Cc_0^7M^2(c_0)+C\int_0^{t}\|\bigtriangledown u\|^2_{H^1}\mathrm{d}\tau.\label{u3}
\end{eqnarray}

Using the elliptic regularity result to the linear movement equation (\ref{dl2}), we can estimate the term $\|\bigtriangledown ^2 u\|_{L^2}$ as follows
\begin{eqnarray}
\|\bigtriangledown  u\|_{H^1}&\leq& C(\|\rho u_t\|_{L^2}+\|\rho v\cdot \bigtriangledown u\|_{L^2}+\|\bigtriangledown p\|_{L^2}+\|\bigtriangledown u\|_{L^2}\nonumber\\
&&+\|(\bigtriangledown d)^{T}(\bigtriangleup d-f(d))\|_{L^2}).\ \label{u4}
\end{eqnarray}
From the assumption (\ref{cs0})-(\ref{cs2}) and the estimates (\ref{md4c}), (\ref{yq3}) and (\ref{lm2a}), we can derive
\begin{eqnarray}
\|(\bigtriangledown d)^{T}\bigtriangleup d\|_{L^2}&\leq& C\|(\bigtriangledown d)^{T}(d_t+v\cdot \bigtriangledown d)\|_{L^2}+ C\|(\bigtriangledown d)^{T}[(n+m)\cdot(d-m)]n\|_{L^2}\nonumber\\
&\leq& C(\|\bigtriangledown d\|_{L^2}^\frac{1}{2}\|\bigtriangledown d\|_{H^1}^\frac{1}{2}\|\bigtriangledown d_t\|_{L^2}+\|\bigtriangledown d\|_{H^1}\|v\|_{H^1}\| \bigtriangledown d\|_{H^1}\nonumber\\
&&+ \|d\|_{H^2}\|n+m\|_{L^6}\|d-m\|_{H^2}\|n\|_{L^6})\nonumber\\
&\leq &Cc_1^6c_0^6,\label{u4a}
\end{eqnarray}
\begin{eqnarray}
\|(\bigtriangledown d)^{T}f(d)\|_{L^2}&\leq& C\|\bigtriangledown d\|_{H^1}\|\|d+m\|_{L^6}\|d-m\|_{L^6}\|d\|_{H^2}\leq Cc_1^4c_0^{12}\qquad\qquad\qquad\label{u4b}
\end{eqnarray}
and
\begin{eqnarray}
&&\|\rho u_t\|_{L^2}+\|\rho v\cdot \bigtriangledown u\|_{L^2}+\|\bigtriangledown p\|_{L^2}+\|\bigtriangledown u\|_{L^2}\nonumber\\
&\leq& \|\sqrt{\rho}\|_{L^\infty}\|\sqrt{\rho} u_t\|_{L^2}+\|\rho\|_{L^\infty}\| v\|_{L^6}\| \bigtriangledown u\|_{L^3}+CM(c_0)c_0+\|\bigtriangledown u\|_{L^2}\nonumber\\
&\leq& Cc_0^\frac{1}{2}\|\sqrt{\rho} u_t\|_{L^2}+Cc_0c_1\| \bigtriangledown u\|^\frac{1}{2}_{L^2}\| \bigtriangledown u\|^\frac{1}{2}_{H^1}+CM(c_0)c_0+\|\bigtriangledown u\|_{L^2}\nonumber\\
&\leq& CM(c_0)c_0+ Cc_1^4(\|\bigtriangledown u\|_{L^2}+\|\sqrt{\rho} u_t\|_{L^2})+\frac{1}{2}\| \bigtriangledown u\|_{H^1}.\label{u4c}
\end{eqnarray}
Taking (\ref{u4a})-(\ref{u4c}) into (\ref{u4}), we can deduce
\begin{eqnarray}
\|\bigtriangledown  u\|_{H^1}\leq  Cc_1^6c_0^{10}+CM(c_0)c_0+ Cc_1^4(\|\bigtriangledown u\|_{L^2}+\|\sqrt{\rho} u_t\|_{L^2}). \label{u4d}
\end{eqnarray}
Taking (\ref{u4d}) into (\ref{u3}) and using Gronwall's inequality, we can derive for all $t\in[0,T_3]$
\begin{eqnarray}
\int_\Omega (\rho |u_t|^2+|\bigtriangledown u |^2)\mathrm{d}x+\int_0^{t}\!\!\!\int_\Omega |\bigtriangledown u_t|^2\mathrm{d}x\mathrm{d}\tau\leq Cc_0^{10}M^2(c_0).\ \label{u5}
\end{eqnarray}
So substituting (\ref{u5}) into (\ref{u4d}), we get
\begin{eqnarray}
\|\bigtriangledown  u\|_{H^1}\leq  Cc_1^6c_0^{10}M(c_0).\label{u6}
\end{eqnarray}
Using the elliptic regularity result to the linear movement equation(\ref{dl2}), the term $\|\bigtriangledown ^2 u\|_{L^6}$ can be controlled as follows
\begin{eqnarray}
\|\bigtriangledown^2 u\|_{L^6}&\leq& C(\|\rho u_t\|_{L^6}+\|\rho v\cdot \bigtriangledown u\|_{L^6}+\|\bigtriangledown p\|_{L^6}+\|(\bigtriangledown d)^{T}(\bigtriangleup d-f(d))\|_{L^6}\nonumber\\
&&+\|\bigtriangledown u\|_{L^6}).\ \label{u7}
\end{eqnarray}
It follows from the assumption (\ref{cs0})-(\ref{cs2}) and the estimates (\ref{md4c}), (\ref{yq3}) and (\ref{u6}) that
\begin{eqnarray}
&&\|\rho u_t\|_{L^6}+\|\rho v\cdot \bigtriangledown u\|_{L^6}+\|\bigtriangledown p\|_{L^6}+\|\bigtriangledown u\|_{L^6}\nonumber\\
&\leq&\|\rho\|_{L^\infty} \|u_t\|_{L^6}+\|\rho\|_{\infty}\| v\|_{L^\infty}\|\bigtriangledown u\|_{L^6}+CM(c_0)c_0+Cc_1^6c_0^{10}M(c_0)\nonumber\\
 &\leq& Cc_0 \|\bigtriangledown u_t\|_{L^2}+Cc_2c_1^6c_0^{11}M(c_0).
\label{u7a}
\end{eqnarray}
It follows from the estimates (\ref{d8}) and (\ref{dd3}) that
\begin{eqnarray}
\|(\bigtriangledown d)^{T}(\bigtriangleup d-f(d))\|_{L^6}
&\leq& C\|d\|^2_{H^3}+C \| \bigtriangledown d\|_{L^6}\|d-m\|_{H^2}\|d+m\|_{H^2}\|d\|_{H^2}\nonumber\\
 &\leq&Cc_2^4c_1^2c_0^6+Cc_1^8 c_0^{12}\nonumber\\
 &\leq&Cc_2^4c_1^4c_0^{12}.\label{u7b}\label{u7c}
\end{eqnarray}
Taking (\ref{u7a}) and (\ref{u7c}) into (\ref{u7}), integrating it over time and using the estimate (\ref{u5}), we can derive for all $t\in[0,T_3],$
\begin{eqnarray}
\int_0^t\|\bigtriangledown^2 u\|^2_{L^6}\mathrm{d}\tau\leq  Cc_0^2\int_0^t \|\bigtriangledown u_t\|^2_{L^2}\mathrm{d}\tau+Cc_0^{18}M^2(c_0)
\leq Cc_0^{18}M^2(c_0).\nonumber
\end{eqnarray}
\end{proof}

It is obvious that Lemma \ref{lm1}-Lemma \ref{lm3} imply Theorem
\ref{j1}.


\subsection{Local existence of a solution to the linear problem (\ref{md2})-(\ref{dl2})}\label{localexlinear}
Since the estimates obtained in the Lemma \ref{lm1}-Lemma \ref{lm3} are
uniform for all small $\delta $, we have the following theorem:

\begin{thm}\label{tma}
 If the initial condition $(\rho_0, u_0, d_0)$ satisfies the regularity $(\ref{zz1})$ and the compatibility condition (\ref{xr1}), then there exists a unique strong solution $(\rho,u,d)$ to the linear equations (\ref{md2})-(\ref{dl2}) with initial boundary value (\ref{cz1})-(\ref{bz1}) such that
\begin{eqnarray}
\begin{array}{ll}
\rho\in C([0,T_3]; W^{1,6}),& \rho_t\in C([0,T_3]; L^6),\\
u\in C([0,T_3]; H_0^1\cap H^2)\cap L^2(0,T_3; W^{2,6}),& u_t\in L^2(0,T_3; H_0^1),\\
d\in C([0,T_3];H^2)\cap L^2(0,T_3; H^3),& d_t\in C([0,T_3]; H_0^1)\cap L^2(0,T_3; H^2),\\
\sqrt{\rho}u_t\in C([0,T_3]; L^2).
\end{array}\label{ze1}
\end{eqnarray}
 Moreover, $(\rho,u,d)$ also satisfies the inequalities (\ref{md4c})-(\ref{yq3}), (\ref{lm2a}) and (\ref{lm3a}).
\end{thm}
Before proof, we give the following two classical lemmas which are proved in the book \cite{temam}.
\begin{lema}\label{ct}
Let $Y=\{v|v\in L^{\alpha_0}(0,T;X_0),v_t\in L^{\alpha_1}(0,T;X_1)\}$\;with norm $|v|_Y=|v|_{L^{\alpha_0}(0,T;X_0)}+|v_t|_{L^{\alpha_1}(0,T;X_1)}$\;where $1<\alpha_0,\alpha_1<\infty,X_0\subset X\subset X_1$\;are Banach spaces and $X_0,X_1$\;are reflexive. Suppose that the injections $X_0\hookrightarrow X\hookrightarrow X_1$\; are continuous, and the injection from $X_0$\;into $X$\;is compact. Then the injection from $Y$\; into $ L^{\alpha_0}(0,T;X)$\; is compact.
\end{lema}
\begin{lema}\label{it}{(An Interpolation Theorem)}
Let $V,\ H, \ V'$ be three Hilbert spaces, each space included in the following one as $V\subset H\equiv H'\subset V'$, $V'$ being the dual of $V$. If a function $u$ belongs to $L^2(0,T; V)$ and its derivative $u_t$ belongs to $L^2(0,T; V')$, then $u$ is almost everywhere equal to a function continuous form $[0,T]$ into $H$.
\end{lema}

\begin{proof} We now start to prove Theorem 2.2.
From the Theorem \ref{j1} and the estimates (\ref{md4c}),
(\ref{yq3}), (\ref{lm2a}) and (\ref{lm3a}), by the compactness Lemma
\ref{ct}, there exists $(\rho, d, u)$ such that
\begin{eqnarray}
(\rho^\delta, d^\delta,u^\delta)&\rightarrow& (\rho, d, u) \ \mathrm{in}\  L^2(0,T_3;L^r\times  H^2\times  W^{1,q}),\label{mdcl}\\
p^{\delta}& \xrightarrow[]{*}& \widehat{p} \,\,\, \mathrm{in}\
L^\infty(0,T_3; W^{1,6}), \quad \mathrm{as}\,\, \delta\rightarrow 0,
\label{yqcl}
\end{eqnarray}
where $\forall r\in (1,+\infty)$ and $ \forall q\in[2,+\infty)$. Hence $p=\widehat{p},\ a.e.$.

Because of the lower semi-continuity of various norms, the estimates (\ref{md4c}), (\ref{lm2a}) and (\ref{lm3a}) also hold for $(\rho,u,d)$. So for almost every $(t,x)( \in  [0,T_3]\times\Omega)$, $(\rho, u, d)$ satisfies the system (\ref{md2})-(\ref{dl2}) which means $(\rho, u,d)$ is a strong solution to the linear equation (\ref{md2})-(\ref{dl2}) with initial-boundary conditions (\ref{cz1})-(\ref{bz1}).

 The solution $(\rho,u,d)$ is unique: From the Lemma
\ref{lm1}, $\rho$ is the unique solution of the linear equation
(\ref{md2}). Using the same method as section \ref{unique}, we can
prove $d$ and $u$ are the unique solution to the linear equations
(\ref{yj2}) and (\ref{dl2}) respectively.

Finally, we will prove the time continuity of the solution $(\rho,
u,d)$. The solution from the Lemma \ref{lm1} is the same as from the
approximation (\ref{mdcl}) due to the uniqueness of solution. So we
get
\begin{eqnarray}
\rho\in C([0,T]; W^{1,6}).\label{rhoc1}
\end{eqnarray}
From the linear equation (\ref{md2}), we easily show
\begin{eqnarray}
\rho_t \in C([0,T_3];L^6).\label{rhoc2}
\end{eqnarray}
By the interpolation Lemma \ref{it}, (\ref{lm3a}) can deduce that
\begin{eqnarray}
 d_t\in L^2(0,T_3;H^2), \quad  d\in L^2(0,T_3;H^3)
\Rightarrow
d\in C([0,T_3];H^2).\label{dc1}
\end{eqnarray}
Differentiating the linear equation (\ref{yj2}) with respect to time and space, we get
\begin{eqnarray}
 &\bigtriangledown d_{tt}+\bigtriangledown(v\cdot \bigtriangledown d)_t=\nu (\bigtriangledown\bigtriangleup d_t-\frac{1}{\sigma^2}\bigtriangledown[(n+m)\cdot (d-m)n]_t).\nonumber
\end{eqnarray}
By the estimate (\ref{lm2a}), we deduce $\bigtriangledown d_{tt} \in L^2(0,T_3;H^{-1})$. Because $\bigtriangledown d_t \in L^2(0,T_3; H^1)$, by the interpolation Lemma \ref{it} again, we have
\begin{eqnarray}
 d_t \in C([0,T_3]; H^1_0).\label{dc2}
\end{eqnarray}
From the linear equation (\ref{yj2}), we get
\begin{eqnarray}
\bigtriangleup d \in C([0,T_3]; H^1).\label{dc3}
\end{eqnarray}
By the interpolation Lemma \ref{it}, (\ref{lm3a}) can deduce that
\begin{eqnarray}
 u_t\in L^2(0,T_3;H^1), \quad u\in L^2(0,T_3;W^{2,6})
\Rightarrow
u\in C([0,T_3];H^1_0).\label{uc1}
\end{eqnarray}
From the linear equation (\ref{dl2}) and the estimates (\ref{md4c}), (\ref{lm2a}) and (\ref{lm3a}), we obtain
\begin{eqnarray}
(\rho u_t, (\rho u_t)_t)\in L^2(0,T_3;H^{1})\times L^2(0,T_3;H^{-1})\Rightarrow\rho u_t \in C([0,T_3]; L^2).\label{uc2}
\end{eqnarray}
From the linear equation (\ref{dl2}) and the elliptic regularity
estimate $|\bigtriangledown^2 u|_{L^2}\leq C|\bigtriangleup
u|_{L^2}$, we have
\begin{eqnarray}
u\in C([0,T_3]; H^2).\label{uc3}
\end{eqnarray}
So we get  time-continuity of $(\rho,u,d)$ from
(\ref{rhoc1})-(\ref{uc3}).
\end{proof}


\section{Iteration and  existence in Theorem 1}\label{It}
Set
\begin{eqnarray}
\begin{array}{cccc}
c_1= Cc_0^{18}M^2(c_0),&c_2=c_1^7,&T_3= \min\{c_2^{-22},T\}.
\end{array}\label{cs3}
\end{eqnarray}
At the beginning, let's choose a initial data of iteration
$(u^0(t,x),d^0(t,x))$. $u^0(t,x)$ satisfies the following heat
equation
\begin{eqnarray}
\phi_t-\bigtriangleup \phi=0,\  \mathrm{with}\  \phi|_{t=0}=u_0,\ \phi|_{\partial \Omega}=0,\nonumber
\end{eqnarray}
and $d^0(t,x)=d_0(x)$. Because $(c_1, c_2,T_3)$ depends only on $c_0$, we can choose a $T_*(\in[0,T_3])$ so small that (\ref{cs1})-(\ref{cs2}) hold for $u^0$ and $d^0$ with $T_*$ instead of $T$.

Replacing $(v,n)$ by $(u^0,d^0)$ and using the Theorem \ref{tma}, we
can obtain  the solution $(\rho^1,u^1,d^1)$ of
(\ref{md2})-(\ref{dl2}) with (\ref{cz1})-(\ref{bz1}), and it
satisfies the estimates (\ref{md4c}), (\ref{lm2a}) and (\ref{lm3a}).
Inductively, for all $k\in N^{+}$, replacing $(v,n)$ by
$(u^{k-1},d^{k-1})$ and using the Theorem \ref{tma}, we can obtain a
sequence $(\rho^k,u^k,d^k)$ of solution of (\ref{md2})-(\ref{dl2})
with (\ref{cz1})-(\ref{bz1}), and they satisfy the following
estimates with the same $(c_0,c_1,c_2,T_*)$ independent of $k\in
N^+$:
\begin{eqnarray}
&&\sup_{0\leq t\leq T_*}(\|u^k\|_{H^1_0}+\|d^k\|_{H^1}+\|d^k_t\|_{H^1_0}+c_1^{-6}
(\|\bigtriangledown ^2 u^k\|_{L^2}+\|\bigtriangledown ^2 d^k\|_{L^2}))\nonumber\\
&&+\int_0^{T_*}\|\bigtriangledown u^k_t\|^2_{L^2}+\|u^k\|^2_{W^{2,6}}+\|\bigtriangledown^2 d_t^k\|^2_{L^2}+\|d^k\|^2_{H^3}\mathrm{d}t\nonumber\\
&& \leq Cc_0^{22}M^2(c_0)\label{gj2}
\end{eqnarray}
and
\begin{align}
\sup_{0\leq t\leq T_*}(\|\rho^k \|_{W^{1,6}}+\|\rho^k_t\|_{L^6})\leq Cc_2c_0,&\sup_{0\leq t\leq T_*}\|\sqrt{\rho^k} u^k_t\|_{L^2}\leq CM(c_0)c_0^{5},\label{gj3}\\
\sup_{0\leq t\leq T_*}(\|p^k\|_{W^{1,6}}+\| p^k_t\|_{L^6}) \leq CM(c_0)c_2c_0,&\sup_{0\leq t\leq T_*}\|\bigtriangledown d^k\|_{H^2}\leq Cc_2^2c_1c_0^3.
\label{gj5}
\end{align}

We will show $(\rho^k, u^k, d^k)$ converges to a strong solution to
the original nonlinear problem (\ref{md1})-(\ref{yj1}).

Define
\begin{eqnarray}
\begin{array}{ccc}
\overline{\rho}^{k+1}=\rho^{k+1}-\rho^{k},&\overline{d}^{k+1}=d^{k+1}-d^{k},&\overline{u}^{k+1}
=u^{k+1}-u^{k}.
\end{array}\label{sc1}
\end{eqnarray}

Since $(\rho^k, u^k, d^k)$ and $(\rho^{k+1}, u^{k+1}, d^{k+1})$
satisfy the linear equations (\ref{md2})-(\ref{dl2}), we have
\begin{align}
&\overline{\rho}^{k+1}_t+\mathrm{div}(\overline{\rho}^{k+1}u^k)+\mathrm{div}(\rho^k \overline{u}^k)=0,\label{md3}\\
&\overline{d}^{k+1}_t-\nu \bigtriangleup \overline{d}^{k+1}=-\overline{u}^k\cdot \bigtriangledown d^{k+1}-u^{k-1}\cdot \bigtriangledown \overline{d}^{k+1}-\frac{\nu}{\sigma^2}((d^k+m)(d^{k+1}-m))\overline{d}^{k}\nonumber\\
& -\frac{\nu}{\sigma^2}((d^k+m)\cdot
\overline{d}^{k+1})d^{k-1}-\frac{\nu}{\sigma^2}(\overline{d}^k(d^{k}-m))d^{k-1},\label{yj3}\\
&\rho^{k+1}\overline{u}^{k+1}_t+\rho^{k+1}u^k\cdot \bigtriangledown\overline{ u}^{k+1}-\mu \bigtriangleup  \overline{u}^{k+1}+\bigtriangledown (p^{k+1}-p^k)\nonumber\\
&=-\lambda (\bigtriangledown \overline{d}^{k+1})^T(\bigtriangleup d^{k+1}-f(d^{k+1}))-\lambda(\bigtriangledown d^k)^T(\bigtriangleup\overline{d}^{k+1}-(f(d^{k+1})-f(d^k)))\nonumber\\
&\ \ \ -\overline{\rho}^{k+1}(u^{k-1}\cdot \bigtriangledown u^k+u_t^k)-\rho^{k+1} \overline{u}^k\cdot \bigtriangledown u^k.\label{dl3}
\end{align}

Define
\begin{eqnarray}
\Psi^{k+1}=\|\overline{\rho}^{k+1}\|^2_{L^2}+\|\overline{d}^{k+1}\|^2_{L^2}+\|\bigtriangledown\overline{d}^{k+1}\|^2_{L^2}
+\|\sqrt{\rho^{k+1}}\overline{u}^{k+1}\|^2_{L^2} .\label{sc2}
\end{eqnarray}

Before estimates, we introduce two small positive unfixed constants $\eta$ and $\epsilon$.

From the first equation (\ref{md3}), we can derive
\begin{eqnarray}
\frac{\mathrm{d}}{\mathrm{d}t}\|\overline{\rho}^{k+1}\|^2_{L^2}\leq \mathbb{A}_\eta^k (t)\|\overline{\rho}^{k+1}\|^2_{L^2}+\eta \|\bigtriangledown \overline{u}^k\|^2_{L^2},\label{scmd4}
\end{eqnarray}
where
 \begin{eqnarray}
 &&\mathbb{A}_\eta^k (t)=C\|\bigtriangledown u^k(t)\|_{W^{1,6}}+\eta ^{-1}C (\|\bigtriangledown \rho ^k(t)\|^2_{L^3}+ \|\rho ^k(t)\|^2_{L^\infty}).\nonumber
 \end{eqnarray}
Using the uniform estimates (\ref{gj2})-(\ref{gj5}), we obtain
 \begin{eqnarray}
 &&\int _0^t\mathbb{A}_\eta^k (s)\mathrm{d}s\leq C+\frac{C}{\eta}t, \qquad\forall t\in[0,T_*]. \label{scmd44}
 \end{eqnarray}

Multiplying (\ref{yj3}) by $ \overline{d}^{k+1} $ and integrating it
over $\Omega$, we have
\begin{align}
&\frac{\mathrm{d}}{\mathrm{d}t}\int_\Omega |\overline{d}^{k+1}|^2\mathrm{d}x+ \int_\Omega|\bigtriangledown\overline{d}^{k+1}|^2\mathrm{d}x\nonumber\\
\leq&\mathbb{B}_\eta^k(t)\|\overline{d}^{k+1}\|^2_{L^2}+\eta (\|\bigtriangledown\overline{u}^k\|_{L^2}^2+ \|\bigtriangledown\overline{d}^k\|_{L^2}^2),\label{yj4}
\end{align}
where for all $t\in[0,T_*],$
 \begin{eqnarray}
& \mathbb{B}_\eta^k (t)
&= C\eta^{-1}\|\bigtriangledown d^{k+1}\|^2_{L^2}+C\|u^{k-1}\|^2_{L^\infty}+C\eta^{-1}\|d^k+m\|^2_{L^6}\|d^k-m\|^2_{L^6}\nonumber\\
&&+C\|d^k+m\|_{L^\infty}\|d^{k-1}\|_{L^\infty}
 +C\eta^{-1}\|d^{k-1}\|^2_{L^6}\|d^k-m\|^2_{L^6}. \nonumber
\end{eqnarray}
The uniform estimates (\ref{gj2})-(\ref{gj5}) implies
 \begin{eqnarray}
&&\int _0^t\mathbb{B}_\eta^k (s)\mathrm{d}s \leq C(1+\frac{1}{\eta})t .\qquad\label{yj44}
 \end{eqnarray}

Multiplying (\ref{yj3}) by $\bigtriangleup \overline{d}^{k+1} $, integrating it over $\Omega$ and using the elliptic estimate $\|\bigtriangledown^2 \overline{d}^{k+1}\|_{L^2}\leq C\|\bigtriangleup \overline{d}^{k+1}\|_{L^2}$ , we can deduce by integration by parts
 \begin{eqnarray}
 &&\frac{\mathrm{d}}{\mathrm{d}t}\int_\Omega |\bigtriangledown\overline{d}^{k+1}|^2\mathrm{d}x+\int_\Omega|\bigtriangledown^2 \overline{d}^{k+1}|^2\mathrm{d}x
 \leq C \sum_{i=1}^{13}L_i,\label{scyj1a}
 \end{eqnarray}
 where
 \begin{eqnarray*}
 L_1&=&\|\bigtriangledown  \overline{u}^k\|_{L^2}\|\bigtriangledown d^{k+1}\|_{L^\infty}\|\bigtriangledown \overline{d}^{k+1}\|_{L^2}\\
 &\leq& C\eta^{-1}\|\bigtriangledown d^{k+1}\|^2_{H^2}\|\bigtriangledown\overline{d}^{k+1}\|_{L^2}^2
  +\eta\|\bigtriangledown\overline{u}^k\|^2_{L^2},\\
  L_2&=&\| \overline{u}^k \|_{L^6}\|\bigtriangledown^2  d^{k+1} \|_{L^3}\| \bigtriangledown  \overline{d}^{k+1}\|_{L^2}\\
  &\leq&C\eta^{-1}\|\bigtriangledown^2   d^{k+1} \|^2_{H^1}\|\bigtriangledown\overline{d}^{k+1}\|_{L^2}^2
   +\eta\|\bigtriangledown\overline{u}^k\|^2_{L^2},\\
  L_3&=&\|\bigtriangledown  u^{k-1}\|_{W^{1,6}}\|\bigtriangledown  \overline{d}^{k+1}\|_{L^2}^2,\\
  L_4&=&\|u^{k-1}\|_{L^\infty}\|\bigtriangledown ^2  \overline{d}^{k+1} \|_{L^2}\| \bigtriangledown  \overline{d}^{k+1}\|_{L^2}\\
     &\leq&C\epsilon^{-1}\|u^{k-1}\|^2_{H^2}\|\bigtriangledown  \overline{d}^{k+1}\|_{L^2}^2+\epsilon\|\bigtriangledown ^2  \overline{d}^{k+1} \|^2_{L^2},\quad\qquad\qquad
\end{eqnarray*}
\begin{eqnarray*}
  L_5&=&\|\bigtriangledown d^k\|_{L^3}\|d^{k+1}-m\|_{L^\infty}\|\overline{d}^k\|_{L^6}\| \bigtriangledown  \overline{d}^{k+1}\|_{L^2}\\
     &\leq&C\eta^{-1}\|\bigtriangledown d^k\|^2_{H^1}\|d^{k+1}-m\|^2_{H^2} \|\bigtriangledown  \overline{d}^{k+1}\|_{L^2}^2+\eta\|\bigtriangledown\overline{d}^k\|^2_{L^2},\\
  L_6&=&\| d^k\|_{L^\infty}\|\bigtriangledown(d^{k+1}-m)\|_{L^3}\|\overline{d}^k\|_{L^6}\| \bigtriangledown  \overline{d}^{k+1}\|_{L^2}\\
     &\leq&C\eta^{-1}\|d^k\|^2_{H^2}\|\bigtriangledown(d^{k+1}-m)\|^2_{H^1} \|\bigtriangledown  \overline{d}^{k+1}\|_{L^2}^2+\eta\|\bigtriangledown\overline{d}^k\|^2_{L^2},\\
  L_7&=&\| d^k\|_{L^\infty}\|d^{k+1}-m\|_{L^\infty}\|\bigtriangledown\overline{d}^k\|_{L^2}\| \bigtriangledown  \overline{d}^{k+1}\|_{L^2}\\
     &\leq& C\eta^{-1}\|d^k\|^2_{H^2}\|d^{k+1}-m\|^2_{H^2} \|\bigtriangledown  \overline{d}^{k+1}\|_{L^2}^2+\eta\|\bigtriangledown\overline{d}^k\|^2_{L^2},
 \end{eqnarray*}
 \begin{eqnarray*}
 L_8&=&\|\bigtriangledown (d^k+m)\|_{L^3}\|\overline{d}^{k+1}\|_{L^6}\|d^{k-1}\|_{L^\infty}\|\bigtriangledown \overline{d}^{k+1}\|_{L^2}\qquad\qquad\\
    &\leq& C\|\bigtriangledown (d^k+m)\|^2_{H^1}\|d^{k-1}\|_{H^2}\|\bigtriangledown \overline{d}^{k+1}\|^2_{L^2},\\
L_9&=&\|d^k+m\|_{H^2}\|d^{k-1}\|_{H^2}\|\bigtriangledown \overline{d}^{k+1}\|^2_{L^2},\\
L_{10}&=&\|d^k+m\|_{L^\infty}\|\overline{d}^{k+1}\|_{L^6}\|\bigtriangledown d^{k-1}\|_{L^3}\|\bigtriangledown \overline{d}^{k+1}\|_{L^2}\\
    &\leq& C\|d^k+m\|^2_{H^2}\|\bigtriangledown d^{k-1}\|_{H^1}\|\bigtriangledown \overline{d}^{k+1}\|^2_{L^2},
 \end{eqnarray*}
 \begin{eqnarray*}
L_{11}&=&\|\bigtriangledown \overline{d}^k\|_{L^2}\|d^{k}-m\|_{L^\infty}\|d^{k-1}\|_{L^\infty}\| \bigtriangledown  \overline{d}^{k+1}\|_{L^2}\\
     &\leq&C\eta^{-1}\|d^k-m\|^2_{H^2}\|d^{k-1}\|^2_{H^2} \|\bigtriangledown  \overline{d}^{k+1}\|_{L^2}^2+\eta\|\bigtriangledown\overline{d}^k\|^2_{L^2},\\
 L_{12}&=&\| \overline{d}^k\|_{L^6}\|\bigtriangledown(d^{k}-m)\|_{L^3}\|d^{k-1}\|_{L^\infty}\| \bigtriangledown  \overline{d}^{k+1}\|_{L^2}\\
     &\leq&C\eta^{-1}\|\bigtriangledown(d^k-m)\|^2_{H^1}\|d^{k-1}\|^2_{H^2} \|\bigtriangledown  \overline{d}^{k+1}\|_{L^2}^2+\eta\|\bigtriangledown\overline{d}^k\|^2_{L^2},\\
 L_{13}&=&\| \overline{d}^k\|_{L^6}\|d^{k+1}-m\|_{L^\infty}\|\bigtriangledown d^{k-1}\|_{L^3}\| \bigtriangledown  \overline{d}^{k+1}\|_{L^2}\\
     &\leq&C\eta^{-1}\|d^k-m\|^2_{H^2}\|\bigtriangledown d^{k-1}\|^2_{H^1} \|\bigtriangledown  \overline{d}^{k+1}\|_{L^2}^2+\eta\|\bigtriangledown\overline{d}^k\|^2_{L^2}.
\end{eqnarray*}
Let's $\epsilon$ small enough, the inequality (\ref{scyj1a}) becomes
  \begin{eqnarray}
 &&\frac{\mathrm{d}}{\mathrm{d}t}\int_\Omega |\bigtriangledown \overline{d}^{k+1}|^2\mathrm{d}x+ \int_\Omega|\bigtriangledown^2 \overline{d}^{k+1}|^2\mathrm{d}x\nonumber\\
  &\leq&C \mathbb{C}_\eta^k(t)|\bigtriangledown\overline{d}^{k+1}|_{L^2}^2
  +C\eta(|\bigtriangledown\overline{d}^k|^2_{L^2}+|\bigtriangledown\overline{u}^k|^2_{L^2}),\label{sca}
  \end{eqnarray}
  where
  \begin{eqnarray}
\mathbb{C}_\eta^k(t)&=&\eta^{-1}|\bigtriangledown d^{k+1}|^2_{H^2} +|\bigtriangledown  u^{k-1}|_{W^{1,6}}+|u^{k-1}|^2_{H^2}+\eta^{-1}|\bigtriangledown d^k|^2_{H^1}|d^{k+1}-m|^2_{H^2}\nonumber\\
 &&+\eta^{-1}|d^k|^2_{H^2}|\bigtriangledown(d^{k+1}-m)|^2_{H^1}+\eta^{-1}|d^k|^2_{H^2}|d^{k+1}-m|^2_{H^2}\nonumber\\
 &&+|\bigtriangledown (d^k+m)|^2_{H^1}|d^{k-1}|_{H^2}+|d^k+m|_{H^2}|d^{k-1}|_{H^2}\nonumber\\
 &&+|d^k+m|^2_{H^2}|\bigtriangledown d^{k-1}|_{H^1}+\eta^{-1}|\bigtriangledown(d^k-m)|^2_{H^1}|d^{k-1}|^2_{H^2}\nonumber\\
 &&+\eta^{-1}|\bigtriangledown(d^k-m)|^2_{H^1}|d^{k-1}|^2_{H^2}+\eta^{-1}|d^k-m|^2_{H^2}|\bigtriangledown d^{k-1}|^2_{H^1}.\nonumber
 \end{eqnarray}
The uniform estimates (\ref{gj2})-(\ref{gj5}) implies
  \begin{eqnarray}
\int _0^t\mathbb{C}_\eta^k (s)\mathrm{d}s\quad\leq C+(C+\frac{C}{\eta})t, \qquad\forall t\in[0,T_*]. \label{yj55}
   \end{eqnarray}

Multiplying the movement equation (\ref{dl3}) by $\overline{u}^{k+1}$ and integrating over $\Omega$, we can deduce
\begin{align}
&\frac{1}{2}\frac{\mathrm{d}}{\mathrm{d}t}\int_\Omega \rho^{k+1}|\overline{u}^{k+1}|^2\mathrm{d}x+\mu\int_\Omega| \bigtriangledown  \overline{u}^{k+1}|^2\mathrm{d}x\nonumber\\
\leq&\int_\Omega|\overline{\rho}^{k+1}|(|u^{k-1}\cdot \bigtriangledown u^k|+|u_t^k|)|\overline{u}^{k+1}|\mathrm{d}x+\int_\Omega|\rho^{k+1}||\overline{u}^k||\bigtriangledown u^k||\overline{u}^{k+1}|\mathrm{d}x\nonumber\\
&+\lambda \int_\Omega|\bigtriangledown \overline{d}^{k+1}||\bigtriangleup d^{k+1}-f(d^{k+1})||\overline{u}^{k+1}|\mathrm{d}x+\lambda£¨\int_\Omega|\bigtriangledown^2 d^k||\bigtriangledown\overline{d}^{k+1} || \overline{u}^{k+1}|\mathrm{d}x\nonumber\\
&+\lambda\int_\Omega|\bigtriangledown d^k||\bigtriangledown\overline{d}^{k+1}||\bigtriangledown \overline{u}^{k+1}|\mathrm{d}x£©+\frac{1}{\sigma^2}\int_\Omega |\bigtriangledown d^k||d^{k+1}+d^k||\overline{d}^{k+1}||d^{k+1}||\overline{u}^{k+1}|\mathrm{d}x \nonumber\\
&+\frac{1}{\sigma^2}\int_\Omega |\bigtriangledown d^k||d^k+m||d^k-m||\overline{d}^{k+1}||\overline{u}^{k+1}|\mathrm{d}x+\int_\Omega| p^{k+1}-p^k||\bigtriangledown \overline{u}^{k+1}|\mathrm{d}x  \nonumber\\
=&\sum_{i=1}^8 M_i,\label{sc3}
\end{align}
where we have used
\begin{eqnarray*}
 f(d^{k+1})-f(d^k)=\frac{1}{\sigma^2}(\overline{d}^{k+1}\cdot(d^{k+1}+d^k)d^{k+1}+(|d^k|^2-1)\overline{d}^{k+1}).
\end{eqnarray*}
Here
\begin{align}
M_1\leq&\|\overline{\rho}^{k+1}\|_{L^2}\|\overline{u}^{k+1}\|_{L^6}(\|u^{k-1}\|_{L^6}\|\bigtriangledown u^k\|_{L^6}+\|u_t^k\|_{L^3})\nonumber\\
\leq&C\|\overline{\rho}^{k+1}\|^2_{L^2}(\|u^{k-1}\|^2_{H^1}\|u^k\|^2_{W^{2,6}}+\|u_t^k\|_{L^2}\|\bigtriangledown u_t^k\|_{L^2})+\frac{\mu}{9}\|\bigtriangledown\overline{u}^{k+1}\|^2_{L^2},\qquad\quad\nonumber\\
M_2\leq&\|\sqrt{\rho^{k+1}}\|_{L^6}\|\overline{u}^k\|_{L^6}\|\bigtriangledown u^k\|_{L^6}\|\sqrt{\rho^{k+1}}\overline{u}^{k+1}\|_{L^2}\nonumber\\
\leq& \eta^{-1}\|\sqrt{\rho^{k+1}}\|^2_{L^6}\| u^k\|^2_{H^2}\|\sqrt{\rho^{k+1}}\overline{u}^{k+1}\|^2_{L^2}+\eta\|\bigtriangledown\overline{u}^k\|^2_{L^2},\nonumber\\
M_3\leq&C\|\bigtriangledown \overline{d}^{k+1}\|_{L^2}\|\bigtriangleup d^{k+1}-f(d^{k+1})\|_{L^3}\|\overline{u}^{k+1}\|_{L^6}\nonumber\\
\leq&\frac{\mu}{9}\|\bigtriangledown\overline{u}^{k+1}\|^2_{L^2}+ C\|\bigtriangledown \overline{d}^{k+1}\|^2_{L^2}(\|\bigtriangledown^2d^{k+1}\|_{L^2}\|\bigtriangledown^2d^{k+1}\|_{L^6}\nonumber\\
&+\|d^{k+1}+m\|^2_{H^1}\|d^{k+1}-m\|^2_{H^1}\|d^{k+1}\|^2_{H^2}),\nonumber
\end{align}
\begin{align}
M_4\leq&C\|\bigtriangledown^2 d^k\|_{L^3}\|\bigtriangledown\overline{d}^{k+1} \|_{L^2}\| \overline{u}^{k+1}\|_{L^6},\nonumber\\
\leq&C\|\bigtriangledown^2 d^k\|_{L^6}\|\bigtriangledown^2 d^k\|_{L^2} \|\bigtriangledown\overline{d}^{k+1} \|^2_{L^2}+\frac{\mu}{9}\|\bigtriangledown\overline{u}^{k+1}\|^2_{L^2} \nonumber\\
M_5\leq&C\|\bigtriangledown d^k\|_{L^\infty}\|\bigtriangledown\overline{d}^{k+1}\|_{L^2}\|\bigtriangledown \overline{u}^{k+1}\|_{L^2}\nonumber\\
\leq&C\|\bigtriangledown d^k\|_{W^{1,6}}^2\|\bigtriangledown\overline{d}^{k+1} \|^2_{L^2}+\frac{\mu}{9}\|\bigtriangledown\overline{u}^{k+1}\|^2_{L^2},\nonumber\\
M_6\leq& C\|\bigtriangledown d^k\|_{L^3}\|d^{k+1}+d^k\|_{L^6}\|\overline{d}^{k+1}\|_{L^6}\|d^{k+1}\|_{L^6}\|\overline{u}^{k+1}\|_{L^6} \nonumber\\
 \leq&C\|\bigtriangledown d^k\|_{L^2}\|\bigtriangledown d^k\|_{H^1}\|d^{k+1}+d^k\|^2_{H^1}\|d^{k+1}\|^2_{H^1}\|\bigtriangledown\overline{d}^{k+1} \|^2_{L^2}\qquad\quad\qquad\quad\quad\nonumber\\
 &+\frac{\mu}{9}\|\bigtriangledown\overline{u}^{k+1}\|^2_{L^2},\nonumber
\end{align}
\begin{align}
M_7\leq&C\|\bigtriangledown d^k\|_{L^3}\|d^k+m\|_{L^6}\|d^k-m\|_{L^6}\|\overline{d}^{k+1}\|_{L^6}\|\overline{u}^{k+1}\|_{L^6}\nonumber\\
\leq& C\|\bigtriangledown d^k\|_{L^2}\|\bigtriangledown d^k\|_{H^1}\|d^k+m\|_{H^1}^2\|d^k-m\|_{H^1}^2\|\bigtriangledown\overline{d}^{k+1} \|^2_{L^2}\qquad\quad\qquad\quad\quad\nonumber\\
&+\frac{\mu}{9}\|\bigtriangledown\overline{u}^{k+1}\|^2_{L^2},\nonumber\\
M_8\leq&\| p^{k+1}-p^k\|_{L^2}\|\bigtriangledown \overline{u}^{k+1}\|_{L^2}  \nonumber\\
\leq& CM^2(c_0)\|\overline{\rho}^{k+1}\|^2_{L^2}
+\frac{\mu}{9}\|\bigtriangledown\overline{u}^{k+1}\|^2_{L^2}.\nonumber
\end{align}
Taking $M_1-M_8$ into (\ref{sc3}), we obtain
\begin{align}
&\frac{\mathrm{d}}{\mathrm{d}t}\int_\Omega \rho^{k+1}|\overline{u}^{k+1}|^2\mathrm{d}x+\int_\Omega| \bigtriangledown  \overline{u}^{k+1}|^2\mathrm{d}x\nonumber\\
\leq&C\mathbb{D}^k_\eta(t)(\|\overline{\rho}^{k+1}\|^2_{L^2}+\|\sqrt{\rho^{k+1}}\overline{u}^{k+1}\|^2_{L^2}+
\|\bigtriangledown \overline{d}^{k+1}\|^2_{L^2})
+C\eta\|\bigtriangledown\overline{u}^k\|^2_{L^2},\label{sc4}
\end{align}
where
\begin{align}
\mathbb{D}^k_\eta(t)=&\|u^{k-1}\|^2_{H^1}\|u^k\|^2_{W^{2,6}}+\|u_t^k\|_{L^2}\|\bigtriangledown u_t^k\|_{L^2}
+\eta^{-1}\|\sqrt{\rho^{k+1}}\|^2_{L^6}\| u^k\|^2_{H^2}\nonumber\\
&+\|\bigtriangledown^2d^{k+1}\|_{L^2}\|\bigtriangledown^2d^{k+1}\|_{L^6}+\|d^{k+1}+m\|^2_{H^1}\|d^{k+1}-m\|^2_{H^1}\|d^{k+1}\|^2_{H^2}\nonumber\\
&+\|\bigtriangledown^2d^{k}\|_{L^6}\|\bigtriangledown^2d^{k}\|_{L^2}+\|\bigtriangledown d^k\|^2_{W^{1,6}}+\|\bigtriangledown d^k\|_{L^2}\|\bigtriangledown d^k\|_{H^1}\nonumber\\
&\cdot(\|d^{k+1}+d^k\|^2_{H^1}\|d^{k+1}\|^2_{H^1}+\| d^k+m\|_{H^1}^2\| d^k-m\|_{H^1}^2)+M^2(c_0).\nonumber
 \end{align}
From the uniform estimates (\ref{gj2})-(\ref{gj5}), we have
  \begin{eqnarray}
&&\int _0^t\mathbb{D}_\eta^k (s)\mathrm{d}s\leq C+(C+\frac{C}{\eta})t, \qquad\forall t\in[0,T_*].\label{sc55}
\end{eqnarray}

Summing (\ref{scmd4}), (\ref{yj4}), (\ref{sca}) and (\ref{sc4}), we obtain
\begin{eqnarray}
&&\frac{\mathrm{d}}{\mathrm{d}t}\Psi^{k+1}+( \|\bigtriangledown\overline{d}^{k+1}\|^2_{L^2}+  \|\bigtriangledown^2\overline{d}^{k+1}\|_{L^2}^2+ \| \bigtriangledown  \overline{u}^{k+1}\|_{L^2}^2)                        \nonumber\\
&\leq&C\mathbb{E}_\eta^k(t)\Psi^{k+1} +C\eta( \|\bigtriangledown \overline{u}^k\|^2_{L^2}+ \|\bigtriangledown\overline{d}^k\|_{L^2}^2 ),\label{sc6a}
\end{eqnarray}
where
\begin{eqnarray}
\mathbb{E}_\eta^k(t)= \mathbb{A}_\eta^k (t)+ \mathbb{B}_\eta^k(t)+ \mathbb{C}_\eta^k(t) +\mathbb{D}^k_\eta(t).\nonumber
 \end{eqnarray}
Using the uniform estimates (\ref{gj2})-(\ref{gj5}), we obtain again
  \begin{eqnarray}
&&\int_0^t \mathbb{E}_\eta^k(s)\mathrm{d}s\leq C +C(1+\eta^{-1})t ,\qquad \forall t\in [0,T_*].
\end{eqnarray}
Applying Gronwall's inequality to (\ref{sc6a}), we can deduce
\begin{eqnarray}
&&\Psi^{k+1}(t)+\int _0^t( \|\bigtriangledown\overline{d}^{k+1}\|^2_{L^2}+  \|\bigtriangledown^2\overline{d}^{k+1}\|_{L^2}^2+ \| \bigtriangledown  \overline{u}^{k+1}\|_{L^2}^2) \mathrm{d}s                       \nonumber\\
&\leq&C\eta\int_0^t( \|\bigtriangledown \overline{u}^k\|^2_{L^2}+ \|\bigtriangledown\overline{d}^k\|_{L^2}^2 )\mathrm{d}\tau\exp(C+C(1+\eta^{-1})t).\label{sc6}
\end{eqnarray}
Hence choose small constants $\eta,\ T^*(<T_*)$, so that
\begin{eqnarray}
C\eta\exp (C+C(1+\eta^{-1})t)\leq \frac{1}{2},\qquad \forall t \in [0,T^*].\label{sc7}
\end{eqnarray}
We easily deduce that
\begin{eqnarray}
&&\sum_{k=1}^{\infty}\sup_{0\leq t\leq T^*}\Psi^{k+1}(t)+\sum_{k=1}^{\infty}\int _0^{T^*}( \|\bigtriangledown\overline{d}^{k+1}\|^2_{L^2}+  \|\bigtriangledown^2\overline{d}^{k+1}\|_{L^2}^2+ \| \bigtriangledown  \overline{u}^{k+1}\|_{L^2}^2) \mathrm{d}s \nonumber\\
&&\leq C<\infty.\label{sc8}
\end{eqnarray}
 (\ref{sc8}) implies that the full sequence $(\rho^k, d^k,u^k)$ converges to a limit $(\rho, d,u)$ in the following strong sense
\begin{eqnarray}
\begin{array}{ll}
\rho^k \rightarrow \rho \ \mathrm{in} \ L^\infty(0,T^*; L^2),&
 u^k\rightarrow u\ \mathrm{in}\ L^2(0,T^*;H^1_0),\\
d^k\rightarrow d\ \mathrm{in}\ L^\infty(0,T^*; H^1)\cap L^2(0,T^*;H^2).&
\end{array} \label{sc9}
\end{eqnarray}
Hence the problem (\ref{md1})-(\ref{yj1}) with initial boundary data
(\ref{cz1})-(\ref{bz1}) has a weak solution $(\rho,u,d)$.
Furthermore using the estimates (\ref{gj2})-(\ref{gj5}), we obtain
that a subsequence of $(\rho^{k}, u^{k},d^{k})$ converges to
$(\rho,u,d)$ in an obvious weak or weak* sense. Due to the lower
semi-continuity of various norms, from (\ref{gj2})-(\ref{gj5}),
$(\rho,u,d)$ also satisfies the following regularity estimates
 \begin{eqnarray}
&&\sup_{0\leq t\leq T^*}(\|\rho \|_{W^{1,6}}+\|\rho_t\|_{L^6}+\|u\|_{H^1_0}+\|p\|_{W^{1,6}}+\| p_t\|_{L^6} +\|d\|_{H^1}+\|d_t\|_{H^1_0}\nonumber\\
&&\qquad\ \ \ \ +\|\bigtriangledown ^2 u\|_{L^2}+\|\bigtriangledown ^2 d\|_{L^2}+\|\bigtriangledown d\|_{H^2})\nonumber\\
&&+\int_0^{T^*}\|\sqrt{\rho} u_t\|^2+\|\bigtriangledown u_t\|^2_{L^2}+\|u\|^2_{W^{2,6}}+\|\bigtriangledown^2 d_t\|^2_{L^2}+\|d\|^2_{H^3}\mathrm{d}t\nonumber\\
&& \leq C.\label{cl2}
\end{eqnarray}
Hence $(\rho,u,d)$ is also a strong solution to the problem
(\ref{md1})-(\ref{yj1}).
\section{Uniqueness and continuity in Theorem 1}\label{unique}
In this section, we will use energy method to prove the uniqueness
and continuity in Theorem 1. For simplicity, we introduce some
notations
\begin{eqnarray}
\begin{array}{lll}
\overline{\rho}=\rho-\widetilde{\rho},&\overline{u}=u-\widetilde{u},&\overline{d}=d-\widetilde{d}.
\end{array}\nonumber
\end{eqnarray}
Define
\begin{eqnarray}
\Psi(t)=\|\overline{\rho}\|^2_{L^2}+\|\overline{d}\|^2_{L^2}+\|\bigtriangledown\overline{d}\|^2_{L^2}+\|\sqrt{\rho}\overline{u}\|^2_{L^2} .\nonumber
\end{eqnarray}
Using the similar process in the section \ref{It}, we can obtain the
following estimate (see (\ref{sc6}))
\begin{eqnarray}
\frac{\mathrm{d}}{\mathrm{d}t}\Psi+( \|\bigtriangledown\overline{d}\|^2_{L^2}+  \|\bigtriangledown^2\overline{d}\|_{L^2}^2+ \| \bigtriangledown  \overline{u}\|_{L^2}^2)\leq C\mathbb{F}(t)\Psi,\label{unq1a}
\end{eqnarray}
where $\mathbb{F}(t)\in L^1(0,T^*)$.\\
Applying the Gronwall's inequality to (\ref{unq1a}), we get for all $t\in [0,T^*],$
\begin{eqnarray}
\Psi(t)+\int_0^{t}(   \|\bigtriangledown^2\overline{d}\|_{L^2}^2+ \| \bigtriangledown  \overline{u}\|_{L^2}^2)\mathrm{d}\tau\leq C\Psi(0),\label{unq1b}
\end{eqnarray}
which implies the uniqueness, and for all $ t\in [0,T^*],$
we have
\begin{eqnarray}
(\|\overline{\rho}\|^2_{L^2}+\|\overline{d}\|^2_{H^1}
+\|\sqrt{\rho}\overline{u}\|^2_{L^2})(t)+\int_0^{t}(   \|\bigtriangledown^2\overline{d}\|_{L^2}^2+ \| \bigtriangledown  \overline{u}\|_{L^2}^2)\mathrm{d}\tau\rightarrow 0\label{unq1e}
\end{eqnarray}
 as $(\widetilde{\rho}_0,\widetilde{u}_0,\widetilde{d}_0)\rightarrow(\rho_0,u_0,d_0)$ in $W^{1,6}\times H^2\times H^3$.

Because $d$ and $\widetilde{d}$ satisfy (\ref{yj1}), we obtain,
similar to (\ref{yj3}),
\begin{eqnarray}
&&\frac{\mathrm{d}}{\mathrm{d}t}\int_\Omega |\bigtriangleup \overline{d}|^2\mathrm{d}x+\int_\Omega |\bigtriangledown \overline{d_t}|^2\mathrm{d}x\nonumber\\
&\leq& C(\|\bigtriangledown \overline{u}\|_{L^2}^2+\|\bigtriangledown\overline{d}\|^2_{L^2}\|\bigtriangledown\widetilde{u}\|^2_{L^6}+\|\bigtriangledown \overline{d}\|_{L^2}^2)+\|\widetilde{u}\|_{L^\infty}^2\|\bigtriangleup \overline{d}\|^2_{L^2}.\nonumber
\end{eqnarray}
Applying Gronwall's inequality to the above inequality,  and using
the inequality (\ref{unq1e}) and the elliptic estimate
$\|\overline{d}\|_{H^2}\leq C \|\bigtriangleup \overline{d}\|_{L^2}
$, we have, for all $ t\in[0,T^*],$
\begin{eqnarray}
\|\overline{d}\|_{H^2}(t)+\int_0^{t} \|\bigtriangledown \overline{d_t}\|^2_{L^2}\mathrm{d}\tau\rightarrow 0\label{unq1f}
\end{eqnarray}
 as $(\widetilde{\rho}_0,\widetilde{u}_0,\widetilde{d}_0)\rightarrow(\rho_0,u_0,d_0)$ in $W^{1,6}\times H^2\times H^3$.

Similarly from (\ref{dl3}), using Gronwall's inequality, (\ref{cl2})
and the convergence (\ref{unq1e})-(\ref{unq1f}), we obtain, for all $ t\in[0,T^*],$
\begin{eqnarray}
\|\overline{u}\|_{H^1}(t)+\int_0^{T^*}\|\sqrt{\rho}\overline{u_t}\|_{L^2}^2\mathrm{d}t\rightarrow 0\label{unq1g}
\end{eqnarray}
as $(\widetilde{\rho}_0,\widetilde{u}_0,\widetilde{d}_0)\rightarrow(\rho_0,u_0,d_0)$ in $W^{1,6}\times H^2\times H^3$.

Multiplying the difference between the continuity equations by $6\overline{\rho}^5$, integrating over $(0,t)\times
\Omega$ and then using Gronwall's inequality, it follows from the
estimate (\ref{cl2}) and the convergence (\ref{unq1g}) that for all $t\in[0,T^*]$,
\begin{eqnarray}
\|\rho-\widetilde{\rho}\|_{L^6}(t)\rightarrow 0.\label{unq1h}
\end{eqnarray}
From the equations (\ref{dl1}) and (\ref{yj1}), by a simple discussion, we can obtain for all $ t\in[0,T^*],$
\begin{eqnarray}
\|\overline{d}\|_{L^2}(t),\  \|\overline{d}\|_{L^2(0,T^*; H^3)},\ \|\overline{u}\|_{L^2(0,T^*; H^2)}\ \ \rightarrow 0 \label{unq1i}
\end{eqnarray}
as $(\widetilde{\rho}_0,\widetilde{u}_0,\widetilde{d}_0)\rightarrow(\rho_0,u_0,d_0)$ in $W^{1,6}\times H^2\times H^3$.

In conclusion, (\ref{unq1e})-(\ref{unq1i}) complete the proof of the
continuity in Theorem 1.


\section{Proof of Theorem \ref{globalex}}

Suppose that there are two positive constants $\theta(<1)$ and
$\widetilde{C}$ such that
\begin{eqnarray}
&&\max\{\|\rho_0\|_{W^{1,6}},\|u_0\|_{H^2},\|d_0-m\|_{H^3}, \|g\|_{L^2}^2\}<\theta,\label{globaex1a}\\
&&\sup_{0\leq t\leq T}(\|v\|_{H^2}+\|n-m\|_{H^2}+\|n_t\|_{H^1_0})
+\int_0^T\|\bigtriangledown v_t\|^2_{L^2}+\|v\|^2_{W^{2,6}}\nonumber\\
&&+\|\bigtriangledown^2 n_t\|^2_{L^2}+\|n\|^2_{H^3}\mathrm{d}t<\widetilde{C}.\label{globalex1b}
\end{eqnarray}
In this section, we assume the genuine constant $C$, maybe depending
on the constant $M(1)$ which is defined by (\ref{yqcs}).

By Lemma \ref{lm1}, there exists a small $\theta_1(< 1)$ so that
$\forall t\in [0,T],\ \forall  \theta\in(0,\theta_1],$
\begin{eqnarray}
\begin{array}{llll}
\|\rho\|_{W^{1,6}}\leq C{\theta}^{\frac{1}{2}},&\|\rho_t\|_{L^6}\leq C{\theta}^{\frac{1}{3}},&
\|p\|_{W^{1,6}}\leq C{\theta}^{\frac{1}{2}},&\|p_t\|_{L^6}\leq  C{\theta}^{\frac{1}{3}},
\end{array}
\label{globalex1c}
\end{eqnarray}
From Lemma \ref{lm2}, we can find a small $\theta_2(< 1)$ so that
$\forall \theta\in(0,\theta_2],$
\begin{eqnarray}
\|d_t\|^2_{H^1}(t),\  \|d-m\|^2_{H^2}(t),\ \int_0^t\|d-m\|^2_{H^3}\mathrm{d}\tau\  \leq C\theta^{\frac{1}{2}},\ \forall t\in [0,T].\label{globalex1d}
\end{eqnarray}
By Lemma \ref{lm3}, a small
$\theta_3(\leq\min\{\theta_1,\theta_2\})$ also can be found so that
$\forall \theta \in (0,\theta_3]$,
\begin{align}
\|u\|^2_{H^2}(t),\ \|\sqrt{\rho}u_t\|^2_{L^2},\ \int_0^t\|u_t\|^2_{H^1}\mathrm{d}\tau,\ \int_0^t\|u\|^2_{W^{2,6}}\mathrm{d}\tau\ \leq C\theta^{\frac{1}{6}},\ \forall t\in [0,T].\label{globalex1e}
\end{align}

Thanks to the estimates (\ref{globalex1c})-(\ref{globalex1e}), using
the Theorem \ref{tma}, we can obtain the global strong solution of
the linear system (\ref{md2})-(\ref{dl2}) with initial boundary
value (\ref{cz1}) and (\ref{bz1}) provided
\begin{eqnarray}
\max\{\|\rho_0\|_{W^{1,6}},\|u_0\|_{H^2},\|d_0-m\|_{H^3}, \|g\|_{L^2}^2\}\leq \theta,\nonumber
\end{eqnarray}
where $\forall\theta\in(0,\theta_3]$.

Now let's talk about the iteration. First, we notice that if $\theta_3$ is taken so small that $C\theta_3\leq \widetilde{C}$, then the process of iteration can be continued for the same $\theta_3$.
Next we will pay attention to the convergence of the iteration.

As the same process in  the section 3, using the estimates
(\ref{globalex1c})-(\ref{globalex1e}), we can obtain,
like (\ref{sc6}),
\begin{align}
&\Psi^{k+1}(t)+\int _0^t( \|\bigtriangledown\overline{d}^{k+1}\|^2_{L^2}+  \|\bigtriangledown^2\overline{d}^{k+1}\|_{L^2}^2+ \| \bigtriangledown  \overline{u}^{k+1}\|_{L^2}^2) \mathrm{d}s                       \nonumber\\
\leq&C\eta\exp(C(t+\theta^\frac{1}{6}t+\theta^\frac{1}{6}+\eta^{-1}\theta^\frac{1}{2}t+\eta^{-1}\theta^\frac{1}{2}))\int_0^t( \|\bigtriangledown \overline{u}^k\|^2_{L^2}+ \|\bigtriangledown\overline{d}^k\|_{L^2}^2 )\mathrm{d}\tau,\label{globalex1f}
\end{align}
where
\begin{eqnarray}
\Psi^{k+1}=\|\overline{\rho}^{k+1}\|^2_{L^2}+\|\overline{d}^{k+1}\|^2_{L^2}+\|\bigtriangledown\overline{d}^{k+1}\|^2_{L^2}
+\|\sqrt{\rho^{k+1}}\overline{u}^{k+1}\|^2_{L^2} .\nonumber
\end{eqnarray}
Hence choose small constants $\eta,\ \theta_0$, so that $\forall \theta\in(0,\theta_0]$ and $\forall t\in[0,T],$
\begin{eqnarray}
C\eta\exp(C(t+\theta^\frac{1}{6}t+\theta^\frac{1}{6}+\eta^{-1}\theta^\frac{1}{2}t+\eta^{-1}\theta^\frac{1}{2}))
\leq \frac{1}{2},\label{globalex1h}
\end{eqnarray}
 We easily deduce that $\forall t\in[0,T],$
\begin{eqnarray}
&&\sum_{k=1}^{\infty}\sup_{0\leq t\leq T}\Psi^{k+1}(t)+\sum_{k=1}^{\infty}\int _0^{T}( \|\bigtriangledown\overline{d}^{k+1}\|^2_{L^2}+  \|\bigtriangledown^2\overline{d}^{k+1}\|_{L^2}^2+ \| \bigtriangledown  \overline{u}^{k+1}\|_{L^2}^2) \mathrm{d}s \nonumber\\
&&\leq C\leq\infty.\label{globalex1i}
\end{eqnarray}
So we complete the proof of Theorem \ref{globalex}.



\end{document}